\documentclass[reqno]{amsart}
\usepackage{yhmath,amsmath,amsfonts,amssymb,enumerate,amsthm,appendix,caption,lipsum,xparse}
\usepackage{enumitem}
\usepackage{mathrsfs}
\usepackage[english]{babel}
\usepackage{hyperref}
\usepackage{cmap,mathtools}
\usepackage{hyperref}
\usepackage{float} 
\usepackage{longtable} 
\usepackage{booktabs}
\usepackage{colortbl}
\usepackage{xcolor}
\usepackage{graphics} 
\usepackage{epsfig} 
\usepackage{graphicx}\usepackage{epstopdf}
\usepackage[pagewise]{lineno}

\usepackage{a4wide}

\usepackage{titlesec} 

\titleformat{\section}{\normalfont\Large\bfseries}{\thesection}{1em}{}

\usepackage[margin=0.95in]{geometry}
\usepackage{cite}
\usepackage[utf8]{inputenc}
\usepackage{xcolor}
\usepackage{hyperref}
\usepackage[hyperpageref]{backref}
\numberwithin{equation}{section} 
\hypersetup{
    colorlinks=true,
    linkcolor=myblue,
    filecolor=magenta,      
    urlcolor=mygreen,
    citecolor=mygreen,
}
\definecolor{mygreen}{rgb}{0.01,0.6,0.2}
\definecolor{myblue}{rgb}{0.01, 0.18, 1.0}
\newtheorem{theorem}{Theorem}
\newtheorem{proposition}[theorem]{Proposition}
\newtheorem{lemma}[theorem]{Lemma}

\theoremstyle{definition}
\newtheorem{definition}[theorem]{Definition}

\numberwithin{equation}{section}
\numberwithin{theorem}{section}
\numberwithin{equation}{section}
\numberwithin{theorem}{section}
\title{Non variational type critical growth nonlocal system}
\author[A. Dixit, H. Hajaiej \& T. Mukherjee]{ASHUTOSH DIXIT$^{1}$, HICHEM HAJAIEJ$^{2}$ and  TUHINA MUKHERJEE$^1$}

\begin{document}
\def\dxy{\mathrm{d}x \mathrm{d}y}
\def\d{\mathrm{d}}
\def\dxz{\mathrm{d}x \mathrm{d}z}
\def\dyz{\mathrm{d}y \mathrm{d}z}
\def\R{\mathbb{R}}
\def\({\bigg(}
\def\){\bigg)}
\def\H{\mathcal{H}^{1,s}(\R^2)}
\def\Hi{\mathcal{H}^{1,s_i}(\R^2)}
\def\Hone{\mathcal{H}^{1,s_1}(\R^2)}
\def\Htwo{\mathcal{H}^{1,s_2}(\R^2)}
\def\ui{\int_{\R^2} \big(|u|^2 + |\partial_x u|^2 + |(-\Delta)^{s_{i}/2}_y u|^2 \big) \dxy}
\def\uone{|u|^2 + |\partial_x u|^2 + |(-\Delta)^{s_{1}/2}_{y} u|^2 }
\def\unone{|u_n|^2 + |\partial_x u_n|^2 + |(-\Delta)^{s_{1}/2}_{y} u_n|^2 }
\def\v2{\int_{\R^2} \big(|v|^2 + |\partial_x v|^2 + |(-\Delta)^{s_{2}/2}_yv|^2 \big) \dxy}
\def\uq{\(\int_{\R^2}|u|^q \dxy\)^{2/q}}
\def\u2si{\(\int_{\R^2}|u|^{2_{s_i}} \dxy\)^{2/2_{s_i}}}
\def\Dely{(-\Delta)_{y}^{s/2}}
\def\Delone{(-\Delta)_{y}^{s_1/2}}
\def\Deltwo{(-\Delta)_{y}^{s_1/2}}
\def\C_c{C_{c}^{\infty}(\R^2)}
\definecolor{dogwoodrose}{rgb}{0.84, 0.09, 0.41}
\definecolor{lava}{rgb}{0.81, 0.06, 0.13}
\def\tb{\color{blue}}
\definecolor{ao(english)}{rgb}{0.0, 0.5, 0.0}
\def\ao{\color{ao(english)}}
\definecolor{robineggblue}{rgb}{0.0, 0.8, 0.8}
\def\tro{\color{tiffanyblue}}
\definecolor{tiffanyblue}{rgb}{0.04, 0.73, 0.71}
\definecolor{tenné(tawny)}{rgb}{0.8, 0.34, 0.0}
\def\ten{\color{tenné(tawny)}}
\definecolor{aqua}{rgb}{0.0, 1.0, 1.0}
\def\aqua{\color{aqua}}
\maketitle 
\centerline{$^{1}$Department of Mathematics, Indian Institute of Technology Jodhpur,}
 \centerline{Rajasthan, 342030, India}
 
 \centerline{$^{2}$Department of Mathematics, California State University,}
\centerline{Los Angeles, CA 90032,
USA} 
\begin{abstract}
This study investigates the existence, uniqueness, and multiplicity of positive solutions for a system of fractional differential equations given by:
\begin{equation*}
(-\Delta)^{s_i} u_{i}+\lambda_{i} u_{i}=\sum_{j=1}^{n} \alpha_{i j}\left|u_{j}\right|^{q_{i j}}\left|u_{i}\right|^{p_{i j}-2} u_{i} , u_i\in {\mathscr{D}^{s_i,2}\left(\mathbb{R}^{N}\right)}, i=1,2,\cdots,n,
\end{equation*}
where $N>2s=\max\{2s_i\}$, $s_i\in(0,1)$, $n\geq 2$, $\lambda_{i} \geq 0$, $\alpha _{ij}>0$, $p_{ij}<2^{*}_{s}$, and $p_{ij}+q_{ij}=2^{*}_{s}=\min\{{\frac{2N}{N-2s_i}\}}$ for $i\neq j \in \{1,2,...,n\}$.  $2^{*}_s$ called the fractional critical sobolev exponent and $2^{*}_s=2 N /(N-2s)$ for $N > 2s$ and $2^{*}_s=+\infty$ for $N=2s$ or $N<2s$. Our work establishes novel uniqueness and multiplicity results for positive solutions, applicable whether the system possesses a variational structure or not. We provide a comprehensive characterization of the exact number of positive solutions under specific parameter configurations. Our analysis shows that the positive solution set behaves differently across three distinct regimes: $p_{ij}<2$, $p_{ij}=2$, and $2<p_{ij}<2^{*}_{s}$. \\\

\end{abstract}

\maketitle
\section*{1. Introduction }
In recent years, the fractional Laplacian and its associated nonlocal integro-differential equations have emerged as powerful tools in modeling complex phenomena across various scientific and mathematical disciplines. These operators have been extensively applied in a variety of fields, including physics, optimization, thin soft films, population dynamics, geophysical fluid dynamics, finance, phase transitions, stratified materials, water waves, game theory, anomalous diffusion, flame propagation, and many others.
 Their developments have been extensively explored in the foundational works of Caffarelli~\cite{caffarelli2012non} and Vázquez~\cite{vazquez2012nonlinear}, among others. The literature on fractional problems is vast and diverse, reflecting extensive research and developments in the field. For a comprehensive understanding of the fundamental properties of the fractional Laplacian,
we direct readers to the following references for further details~\cite{bhaktanonhomogeneous,fiscella2018p,mingqi2018combined,bisci2015ground,MR3271254,xiang2019superlinear,dihich}. Nonlocal operators like the fractional Laplacian and its nonlinear counterpart, the fractional $p$-Laplacian, extend classical local operators such as the Laplacian $-\Delta$ and $p$-Laplacian $-\Delta_p$. These operators enable generalizations of foundational elliptic problems, including the renowned Bre\'{z}is-Nirenberg problem.

The inherent nonlocality of the fractional Laplacian introduces significant analytical challenges. To address this, Caffarelli and Silvestre pioneered a transformative approach in their seminal work~\cite{caffarelli2007extension}, known as the extension method. This technique reduced nonlocal problems into higher-dimensional spaces, effectively converting them into local formulations. Building on this foundation, contemporary research has focused on mixed operator systems, normalized solutions, and multi-component variational problems \cite{chergui2023existence,luo2022normalized,al2018existence,hajaiej2023normalized,huang2025lazer}. Although E. Abada et al. \cite{abada2021topological} pioneered Leray-Schauder degree theory for two-component fractional systems, their framework collapses for $K\geq 3$ components under critical exponents. This leaves a critical void: no unified non-variational existence theory exists for $K\geq 3$ fractional systems with critical nonlinearities, a gap our work resolves through topological degree adaptations.

\section*{Foundational Results}
In 1983, Brezis and Nirenberg~\cite{brezis1983positive} demonstrated the existence of classical solutions for the critical semilinear problem:
\[
\begin{cases}
-\Delta u = \lambda u + u|u|^{2^*-2}, & \text{in } \Omega, \\
u = 0, & \text{on } \partial \Omega, \\
u > 0, & \text{in } \Omega,
\end{cases}
\]
 when $\lambda \in (0, \lambda_1)$ and $N \geq 4$, where $\lambda_1$ denotes the principal eigenvalue of $-\Delta$ under Dirichlet conditions, $\Omega$ is a bounded domain in $\mathbb{R}^N$ and $2^* = \frac{2N}{N-2}$ is the critical Sobolev exponent. For $N=3$, solutions exist if $\mu < \lambda < \lambda_1$, for a suitable $\mu>0$ (with $\mu = \frac{1}{4}\lambda_1$ if $\Omega$ is a ball). The Pohozaev identity precludes solutions for $\lambda \notin (0, \lambda_1)$ in star-shaped domains.
In 1985, A. Capozzi et al. \cite{capozzi1985existence} later extended these results, proving nontrivial solutions exist for $\lambda > 0$ when $N \geq 4$.
Subsequently, In 1986 Ambrosetti and Struwe~\cite{ambrosetti1986note} employed a dual formulation for problem, enabling direct application of the Mountain-Pass Theorem and critical point framework established by Ambrosetti and Rabinowitz~\cite{ambrosetti1973dual,rabinowitz1986minimax}. This approach yielded a more concise demonstration of nontrivial solution existence for the problem.
In 2016, F.Gladiali et al. \cite{gladiali2016non} studied the following non-variational system
\[
\begin{cases}
-\Delta u_i = \sum_{j=1}^{k} \alpha_{i j}\left|u_{j}\right|^{\frac{N+2}{N-2}}, & \text{in } \mathbb{R}^N, \\
u_i >0, & \text{in }\mathbb{R}^N, \\
u_i\in \mathscr{D}^{1,2}\left(\mathbb{R}^{N}\right),
\end{cases}
\]
The authors establish structural conditions on the matrix $(\alpha_{ij})_{i,j=1}^k$ that guarantee bifurcation of solutions emerging from the critical Sobolev equation.

\section*{Modern Generalizations}
In 2014, Fei Fang \cite{fang2014infinitely} tackle the following fractional Laplacian problem with pure critical nonlinearity
\[
\begin{cases}
(-\Delta)^s u = |u|^{2_s^*-2}u, & \text{in } \mathbb{R}^N, \\
u\in \mathcal{D}^{s,2}(\mathbb{R}^N),
\end{cases}
\]
where $s\in(0,1)$, $2^*=\frac{2N}{N-2s}$, $N$ is a positive integer with $N\geq 3$. They have proved the above problem  has infinitely many non radial sign changing solution.
In 2015, Servadei et al.~\cite{MR3271254} generalized the Brezis-Nirenberg framework to fractional settings:
\[
\begin{cases}
(-\Delta)^s u = \lambda u + u|u|^{2_s^*-2}, & \text{in } \Omega, \\
u = 0, & \text{in } \mathbb{R}^{N} \setminus \Omega,
\end{cases}
\]
where $s \in (0,1)$ and $2_s^* = \frac{2N}{N-2s}$. For $\lambda \in (0, \lambda_{1,s})$ (with $\lambda_{1,s}$ as the principal eigenvalue of $(-\Delta)^s$), nontrivial solutions exist when $N \geq 4s$.
In 2018, J. F. Bonder et al. \cite{bonder2018concentration} extends the well known concentration compactness principle for the Fractional Laplacian in unbounded domain. They have considered the quasi linear fractional Laplacian with critical nonlinearities. 
Now the system of fractional Laplacian in 2017 was studied by Li Wang et al. \cite{wang2017fractional}. They have  established the existence of solution of fractional Laplacian system involving critical nonlinearities using variational method.

In 2021, E. Abada et al. \cite{abada2021topological} discussed the following problem 
\[
\begin{cases}
(-\Delta)^s u(x) + g_1(x,u(x),v(x))=f_1(x) , & \text{in } \Omega, \\
(-\Delta)^s v(x) + g_2(x,u(x),v(x))=f_2(x), & \text{in } \Omega, \\
u=v=0  & \text{in } \mathbb{R}^{N} \setminus \Omega,
\end{cases}
\]
where $s\in(0,1)$, $\Omega$ is a bounded open subset of $\mathbb{R}^N$ with Lipschitz boundary, and $(f_1,f_2)\in L^2(\Omega)\times L^2(\Omega)$ and $g_1,g_2:\Omega\times \mathbb{R}\times \mathbb{R}\to \mathbb{R}$ are satisfying the caratheodary conditions. They employed a non-variational approach based on Leray-Schauder degree theory to establish the existence of solutions for the problem.

The system in \eqref{eq:1.3b} is significantly more intricate than single equations, with its complexity growing as the number of equations increases. Researchers have discovered several distinctive properties of solutions that are absent in single-equation cases. These include the existence and multiplicity of nontrivial solutions, the segregation and synchronization of solution components, and the nodal behavior of solutions. For further details, refer to studies on the subcritical case \cite{ambrosetti2007standing,bartsch2010liouville,byeon2020positive,byeon2021partly,dancer2010priori,peng2013segregated}, and the critical case \cite{guo2017critical,li2018multiple,chen2015positive,clapp2018existence,clapp2019simple,peng2016elliptic,tavares2020existence,lu2020existence}.
Most existing studies on fractional Laplacian systems focus on problems with a built-in variational structure, meaning they can be analyzed using energy minimization techniques. These studies typically rely on variational methods, which work well for system of equations. However, a major limitation remains—there’s very little known about fractional systems with three or more components by using non variational technique. This gap leaves open critical questions about how solutions behave in more complex, multi-equation setups, which our work aims to address.
      
We consider the following system of $n$-coupled equations
\begin{equation}\label{eq:1.1}
(-\Delta)^{s_i} u_{i}+\lambda_{i} u_{i}=\sum_{j=1}^{n} \alpha_{i j}\left|u_{j}\right|^{q_{i j}}\left|u_{i}\right|^{p_{i j}-2} u_{i} \quad \text { in } \Omega, \quad i=1,2, \cdots, n
\end{equation}
where $\Omega = \mathbb{R}^N$ with $N >2s=\max\{2s_i\}$, $n \geq 2$, $\lambda_{i} \geq 0$ for $i=1,2, \cdots$, $n$,  $\alpha_{i j}>0,$  $q_{i j}>0$, and the exponent $p_{i j}+q_{i j}$ satisfying the relation
\[
p_{i j}+q_{i j}=r  ~~ \text{for some r }\in\left(2,2^{*}_s\right]~~ \text{and for} ~~i, j=1,2, \cdots, n.
\]
The term ${(- \Delta)^{s}}$ denotes the fractional Laplace operator, which is for a fixed parameter $s \in (0,1)$ defined by 
$${(- \Delta)^{s}u(x)} = C(N,s) ~\text{P.V.}\int_{\mathbb{R}^N} {\dfrac{u(x)-u(y)}{|x-y|^{N+2s}}} ~ dy $$ 
where the term “P.V.” stands for Cauchy’s principal value, while C(N,s) is a normalizing constant whose explicit expression is given by$$C(N,s)= \bigg( \int_{\mathbb{R}^{N}} {\frac{1-cos({\zeta}_{1})}{|\zeta|^{N+2s}}}~d \zeta \bigg)^{-1}$$

The system \eqref{eq:1.1} possesses a variational structure if and only if $p_{i j}=q_{j i}$ and $\frac{\alpha_{i j}}{p_{i j}}=\frac{\alpha_{j i}}{p_{j i}}$ for all $i \neq j$. In particular, if $p_{i j}=q_{i j}=\frac{r}{2}$ and $\alpha_{i j}=\alpha_{j i}$ then the corresponding energy functional is
\begin{equation}\label{eq:1.2}
J\left(u_{1}, \cdots, u_{n}\right)=\frac{1}{2} \sum_{i=1}^{n} \int_{\mathbb{R}^N}\left(\int_{\mathbb{R}^N}\frac{|u(x)-u(y)|^2}{|x-y|^{N+2s_i}}+\lambda_{i} u_{i}^{2}\right)-\frac{1}{r} \sum_{i, j=1}^{n} \alpha_{i j} \int_{\mathbb{R}^N}\left|u_{i} u_{j}\right|^{\frac{r}{2}}.
\end{equation}

In this paper, we consider system \eqref{eq:1.1} with the fractional critical Sobolev exponent case: $p_{i j}+q_{i j}=2^{*}_s$, and we denote $\eta_{i}:=\alpha_{i i}$ and assume $\lambda_{i}=0$, $\Omega=\mathbb{R}^{N}$ and $N > 2s$ in \eqref{eq:1.1}. This leads us to study the system
\begin{equation}\label{eq:1.3b}
\left\{\begin{array}{l}
(-\Delta)^{s_i} u_{i}=\eta_{i} u_{i}^{2^{*}_s-1}+\sum_{j=1, j \neq i}^{n} \alpha_{i j} u_{i}^{p_{i j}-1} u_{j}^{q_{i j}} \quad ,\text { in } \mathbb{R}^{N}, \\
u_{i}>0, \quad u_{i} \in \mathscr{D}^{s_i,2}\left(\mathbb{R}^{N}\right), \quad i=1,2, \cdots, n.
\end{array}\right.
\end{equation}
where the space $D^{s_i,2}(\mathbb{R}^N)$ is defined as  the completion of $C^{\infty}_0(\mathbb{R}^N)$ under the norm 
$$\|u\|_{D^{s_i,2}(\mathbb{R}^N)}:=\int_{\mathbb{R}^N}|(-\Delta )^{\frac{s_i}{2}}u(x)|^2\mathrm{d}x$$

Throughout the paper, we assume the following-
\begin{itemize}
\item[(A1)] $\eta_{i}>0$ and non decreasing i.e. $\eta_{1} \leq \eta_{2} \leq \cdots \leq \eta_{n}.$
\item[(A2)] $\alpha_{i j}>0$
\item[(A3)] $p_{i j}+q_{i j}=2^{*}_s \text { for } i \neq j \in\{1,2, \cdots, n\}$. 
\end{itemize}
If $p_{ij}<2^*_s$, we have $q_{i j}>0$. However, we allow $p_{i j} \leq 0$ and accordingly $q_{i j} \geq 2^{*}_s$ in some of our results.
\definition
A vector solution $\left(u_{1}, \cdots, u_{n}\right)$ is said to be nontrivial if every component $u_{i}$ is nonzero. In contrast, semitrivial solutions have at least one component equal to zero and at least one component that is nonzero. In this work, we focus on positive solutions, which are nontrivial solutions $\left(u_{1}, u_{2}, \cdots, u_{n}\right)$ satisfying $u_{i}(x) > 0$ for all $i = 1, 2, \cdots, n$ and for all $x \in \mathbb{R}^{N}$.

\definition
A solution $\left(u_{1}, u_{2}, \cdots, u_{n}\right)$ of \eqref{eq:1.3b} is called a synchronized solution if there exist positive numbers $k_{1}, k_{2}, \cdots, k_{n}$ such that $u_{i}=k_{i} U$ for $i=1,2, \cdots, n$, where $U$ is a positive solution of the single equation
\begin{equation}
(-\Delta)^s u=u^{2^{*}_s-1} \quad \text { in } \mathbb{R}^{N}, \quad u \in \mathscr{D}^{s,2}\left(\mathbb{R}^{N}\right)
\end{equation}

According to \cite[Remark 1.1]{fang2014infinitely}, $U$ is unique up to translation and dilation and has the expression
\begin{equation}\label{eq:1.5}
U(x)=\frac{[N(N-2s)]^{\frac{N-2s}{4s}}}{\left(1+|x|^{2}\right)^{\frac{N-2s}{2}}}
\end{equation}

\section{The main results}
The results of this paper is  on uniqueness, multiplicity or exact multiplicity of solutions of \eqref{eq:1.3b} all up to translation and dilation. Since we do not assume any symmetry condition, the system is essentially more general than those having variational structure. For convenience, denote $B=\left(\alpha_{i j}\right)_{n \times n}$.
Our first result deals with the case \textbf{$p_{i j}<2$} and gives existence of $2^{n}-1$ synchronized positive solutions if $\alpha_{i j}(i \neq j)$ are suitably small and existence of one synchronized positive solution if $p_{i j}=p$ for $i \neq j$.
\begin{theorem}\label{thm:2.1}
Assume that $N >2s$ and $ p_{i j}<2$ along with $(A1), (A2)$ and $(A3)$ then
\begin{enumerate}
\item[(a)] If $\alpha_{i j}$ satisfies 
\[
0<\alpha_{i j}<\alpha_{*}:=\frac{1}{2} \min _{1 \leq i \leq n}\left[\sum_{j=1, j \neq i}^{n}\left(\frac{1}{\eta_{j}}\right)^{\frac{N-2s}{4s} q_{i j}}\left(\frac{1}{2 \eta_{i}}\right)^{\frac{N-2s}{4s}\left(p_{i j}-2\right)}\right]^{-1},
\]
then \eqref{eq:1.3b} has at least $2^{n}-1$ synchronized positive solutions.
\item[(b)] If  $p_{i j}=p<2$(constant) for $i \neq j$ and the matrix $B=\left(\alpha_{i j}\right)_{n \times n}$ has an inverse $A=\left(a_{i j}\right)_{n \times n}$ such that
\[
a_{i j}>0 \text{ for } i \neq j \text{ and } \sum_{j=1}^{n} a_{i j}>0 \text{ for } i=1,2, \cdots, n,
\]
then \eqref{eq:1.3b} has at least one synchronized positive solution.
\end{enumerate}
\end{theorem}
Following results concerns about exact number of solutions to \eqref{eq:1.3b} under restrictive assumption of $\alpha_{ij}$'s are constant.
\begin{theorem}\label{thm:2.2}
Assume that $N >2s$, $ p_{i j}=p<2$ along with $(A1), (A2)$ and $(A3)$ and $\alpha_{i j}=\alpha$ for $i, j \in\{1,2, \cdots, n\}$ such that $i \neq j$. Then
\begin{enumerate}
\item[(a)]  \eqref{eq:1.3b} has at least one synchronized positive solution, for any $\alpha>0$.
\item [(b)]If either  $\alpha \geq \eta_{n}$, or  $\alpha \geq \eta_{n}-\gamma_{0}$ when $\eta_{n-1}<\eta_{n}$ where $\gamma_{0}$ is some positive number, then  \eqref{eq:1.3b} has exactly one synchronized positive solution.
\item[(c)] There exists $\alpha_{0} \in\left(0, \eta_{1}\right)$ such that  \eqref{eq:1.3b} has exactly $2^{n}-1$ synchronized positive solutions for $0<\alpha<\alpha_{0}$.
\item[(d)] If $n=2 m, \eta_{1}=\cdots=\eta_{m}=: \eta^{\prime} \leq \eta_{m+1}=\cdots=\eta_{2 m}=: \eta^{\prime \prime}$, $s_1=s_2=\cdots=s_{2m}=:s^{\prime},$ and
\[
\alpha>\frac{(m+1) \eta^{\prime \prime}-(m-1) \eta^{\prime}+\sqrt{(m+1)^{2} \eta^{\prime \prime 2}+(m-1)^{2} \eta^{\prime 2}-2\left(m^{2}+1\right) \eta^{\prime} \eta^{\prime \prime}}}{2},
\]
then  \eqref{eq:1.3b} has exactly one positive solution.
\end{enumerate}
\end{theorem}

We now consider the case \textbf{$p_{ij}=2$} and present existence, non existence, uniqueness and exact multiplicity for \eqref{eq:1.3b} under various assumptions.
\begin{theorem}\label{thm:2.3}
Assume that $N >2s$, $ p_{i j}=2$ along with $(A1), (A2)$ and $(A3)$  and $\alpha_{i j}=\alpha$ for $i, j \in\{1,2, \cdots, n\}$ such that $i \neq j$. Then
\begin{enumerate}
\item[(a)] \eqref{eq:1.3b} has a synchronized positive solution if and only if $\alpha>\eta_{n}$ or $0<\alpha<\eta_{1}$ or $\alpha=\eta_{1}=\eta_{n}$. Moreover, if $\alpha>\eta_{n}$ or $0<\alpha<\eta_{1}$ then \eqref{eq:1.3b} has exactly one synchronized positive solution and if $\alpha=\eta_{1}=\eta_{n}$ then \eqref{eq:1.3b} has infinitely many synchronized positive solutions.
\item[(b)] If $\eta_{1} \leq \alpha \leq \eta_{n}$ and $\eta_{1} \neq \eta_{n}$ then \eqref{eq:1.3b} has no positive solution.
\item[(c)] If $\alpha>\eta_{n}$ then \eqref{eq:1.3b} has exactly one positive solution.
\end{enumerate}
\end{theorem}

Lastly, we prove our results for the case $2<p_{ij}<2^*_s$.
\begin{theorem}\label{thm:2.4}
Assume that $N >2s$, $2<p_{i j}<2^{*}_{s}$ along with $(A1), (A2)$ and $(A3)$  then
 \eqref{eq:1.3b} has a synchronized positive solution.
\end{theorem}

\begin{theorem}\label{thm:2.5}
Assume that $N >2s$, $(A1), (A2)$ and $(A3)$, $p_{i j}=p \in\left(2,2^{*}_{s}\right), \alpha_{i j}=\alpha$ for $i, j \in\{1,2, \cdots, n\}$ and $i \neq j$. Then
\begin{itemize}
\item[(a)] If either $0<\alpha \leq \eta_{1}$, or  $0<\alpha \leq \eta_{1}+\gamma_{0}$ when $\eta_{1}<\eta_{2}$, where $\gamma_{0}$ is some positive number, then \eqref{eq:1.3b} has exactly one synchronized positive solution.
\item[(b)] If $\alpha>\eta_{j}$ then there exists $p_{1}=p_{1}(\alpha) \in\left(2,2^{*}_{s}\right)$ such that for $p \in\left(2, p_{1}\right), \eqref{eq:1.3b}$ has at least $2^{j}-1$ synchronized positive solutions. In particular, if $\alpha>\eta_{n}$ then for $p$ larger than and sufficiently close to 2 , \eqref{eq:1.3b} has at least $2^{n}-1$ synchronized positive solutions.
\item[(c)] If $\alpha>\eta_{1}$ and
\[
\frac{\eta_{1}}{\alpha} 2+\left(1-\frac{\eta_{1}}{\alpha}\right) 2^{*}_{s} \leq p<2^{*}_{s},
\]
then \eqref{eq:1.3b} has exactly one synchronized positive solution. In particular, this result together with (a) implies that for any $\alpha>0$, \eqref{eq:1.3b} has exactly one synchronized positive solution if $p$ is less than and sufficiently close to $2^{*}_{s}$.
\end{itemize}
\end{theorem}
Now we state some remarks which illustrates few more facts and consequences of our above main results.

{\bf{Remark 2.1}}\label{rem:1.1}
It is worth noting that Theorem~\ref{thm:2.1} encompasses the previously unexplored case where $p_{ij} \leq 0$. For Theorem~\ref{thm:2.1}(a), we require that the values of $\alpha_{ij}$ (when $i \neq j$) remain sufficiently small. Section 3's proof demonstrates that we can enhance $\alpha_{*}$ to
\[
\alpha_{**} = \max_{0<\gamma<1} f(\gamma)
\]
with the function $f$ defined as
\[
f(\gamma) = (1-\gamma) \min_{1 \leq i \leq n}\left[\sum_{j=1, j \neq i}^{n}\left(\frac{1}{\eta_{j}}\right)^{\frac{N-2s}{4s} q_{ij}}\left(\frac{\gamma}{\eta_{i}}\right)^{\frac{N-2s}{4s}\left(p_{ij}-2\right)}\right]^{-1}
.\]

One can observe that $\alpha_{**} \geq f(1/2) = \alpha_{*}$. The matrix condition B described in Theorem~\ref{thm:2.1}(b) can be fulfilled in two scenarios: first, when all $\alpha_{ij}$ values ($i \neq j$) approximate a common value $\alpha$ exceeding $\eta_{n}$ (as detailed in Proposition~\ref{pro:2.1}); second, when the $\alpha_{ij}$ values are clustered into groups with elements in each group approximating a sufficiently large common value $\alpha$ (elaborated in Proposition~\ref{pro2.2}). Essentially, Theorem~\ref{thm:2.1}(b) indicates that equation~\eqref{eq:1.3b} admits at least one synchronized positive solution when $p_{ij} = p$ for all $i \neq j$ and the corresponding $\alpha_{ij}$ values are adequately large.

{\bf{Remark 2.2}}\label{remark1.2}
Parts (a)-(c) of Theorem~\ref{thm:2.2} address synchronized positive solutions exclusively. It
would be interesting to prove that the number of synchronized positive solutions is decreasing with respect
to  $\alpha>0$.

{\bf{Remark 2.3}}\label{remark1.3}
While Theorem~\ref{thm:2.2} operates under more restrictive assumptions, its findings offer considerably more refined insights compared to those presented in Theorem~\ref{thm:2.1}. Specifically, Theorem~\ref{thm:2.2} guarantees the existence of synchronized positive solutions across the entire domain where $\alpha>0$, and furthermore provides precise enumeration of such synchronized positive solutions in both asymptotic regimes—when $\alpha>0$ is sufficiently large and when $\alpha>0$ approaches zero. Additionally, under enhanced structural properties of the matrix $B$, Theorem~\ref{thm:2.2} establishes the uniqueness criterion for positive solutions to equation~\eqref{eq:1.3b} in the regime where $\alpha>0$ is adequately large.

{\bf{Remark 2.4}}\label{remark1.4}
Our findings in Theorems~\ref{thm:2.1} and~\ref{thm:2.2} extend beyond previous research in several important ways. The system~\eqref{eq:1.3b} we study generalizes earlier systems that were typically analyzed using variational methods. Importantly, variational approaches cannot be applied to prove our theorems because system~\eqref{eq:1.3b} may not possess variational structure. We specifically include the previously unstudied cases where $p_{ij} \leq 0$ and $q_{ij} \geq 2^{*}_{s}$. For reference, system~\eqref{eq:1.3b} has variational structure only when $p_{ij}=q_{ji}$ and $\frac{\alpha_{ij}}{p_{ij}}=\frac{\alpha_{ji}}{p_{ji}}$ for all $i \neq j$.

Our work establishes several new results regarding solution multiplicity which we summarize below:\\
\begin{table}[H]
\centering
\begin{tabular}{p{6cm}p{6cm}}
\toprule
\rowcolor{gray!20}\textbf{Parameter Condition} & \textbf{Number of Synchronized Positive Solutions} \\
\midrule
$\alpha_{ij}$ values sufficiently small & At least $2^{n}-1$ solutions \\
& (Theorem~\ref{thm:2.1}(a)) \\
\midrule
Sufficiently small $\alpha$ & Exactly $2^{n}-1$ solutions \\
& (Theorem~\ref{thm:2.2}(c)) \\
\midrule
$\alpha$ exceeds certain thresholds & Exactly one solution \\
& (Theorem~\ref{thm:2.2}(b)) \\
\bottomrule
\end{tabular}
\caption{Theorem Results on Synchronized Positive Solutions}
\label{tab:theorem_results}

\end{table}

These results advance our understanding even for systems with variational structure, providing insight into how the solution bifurcation diagram relates to parameters $n$ and $\alpha$. Additionally, Theorem~\ref{thm:2.2}(d) provides a new uniqueness criterion for all positive solutions representing a significant advancement in the analysis of critical elliptic systems.

{\bf{Remark 2.5}}\label{remark1.5}
Looking at Theorems~\ref{thm:2.2} and~\ref{thm:2.3}, we can see that solutions behave very differently depending on certain key values.

\begin{table}[H]
\centering
\begin{tabular}{p{6cm}p{6cm}}
\toprule
\rowcolor{gray!20} \textbf{Parameter Condition} & \textbf{Solution Properties} \\
\midrule
When $p_{ij}=p<2$ & Solutions exist for any positive value of $\alpha$ (no matter how large or small) \\
\midrule
When $p_{ij}=2$ & No solutions exist when $\alpha$ is between $\eta_1$ and $\eta_n$ (assuming these values are different) \\
\midrule
For small values of $\alpha$ with $p_{ij}=p<2$ & Exactly $2^n-1$ different solutions exist \\
\midrule
For small values of $\alpha$ with $p_{ij}=2$ and $\alpha<\eta_1$ & Only one solution exists \\
\midrule
\end{tabular}
\caption{Existence and Multiplicity of Synchronized Positive Solutions}
\label{tab:solutions}
\end{table}

These big differences in how solutions behave suggest  a "bifurcation phenomenon" happens as the value of $p_{ij}$ approaches 2 where solution patterns change significantly as a parameter value crosses a threshold.

{\bf{Remark 2.6}}\label{remark1.6}
The assumptions established in Theorem~\ref{thm:2.3} inherently prevent the system from exhibiting variational structure, thereby rendering variational methodologies inapplicable to this context. Moreover, Theorem~\ref{thm:2.3} delivers dual analytical contributions: it precisely enumerates the synchronized positive solutions as determined by the magnitude of parameter $\alpha$, while simultaneously establishing the definitive uniqueness of all positive solutions in scenarios where $\alpha$ exceeds the threshold value $\eta_n$.

{\bf{Remark 2.7}}\label{remark1.7}
Theorem~\ref{thm:2.4} and Theorem~\ref{thm:2.5} shows that when $2 < p_{ij} < 2^{*}_{s}$, the equation system~\eqref{eq:1.3b} behaves very differently from what we saw in Theorems~\ref{thm:2.1}-\ref{thm:2.3} (where $p_{ij} < 2$ or $p_{ij} = 2$).

Let's look at the special case where all $p_{ij} = p$ and all $\alpha_{ij} = \alpha$ to see these differences clearly:


\begin{table}
\centering
\renewcommand{\arraystretch}{1.5}
\begin{tabular}{>{\raggedright\arraybackslash}p{7cm}|>{\raggedright\arraybackslash}p{7cm}}
\toprule
\rowcolor{gray!20} \textbf{Parameter Range} & \textbf{Number of Synchronized Positive Solutions} \\
\midrule
\multicolumn{2}{l}{\textbf{When $\boldsymbol{\alpha}$ is small:}} \\
\hline
$p < 2$ & Exactly $2^n-1$ solutions \\
\hline
$2 \leq p < 2^*_s$ & Exactly 1 solution \\
\midrule
\multicolumn{2}{l}{\textbf{When $\boldsymbol{\alpha > \eta_n}$:}} \\
\hline
$p \leq 2$ & Exactly 1 solution \\
\hline
$p$ slightly larger than 2 & At least $2^n-1$ solutions \\
\midrule
\multicolumn{2}{l}{\textbf{For any $\boldsymbol{\alpha > 0}$:}} \\
\hline
$p \neq 2$ and $p < 2^*_s$ & At least 1 solution \\
\hline
$p = 2$ and $\eta_1 \leq \alpha \leq \eta_n$ (with $\eta_1 \neq \eta_n$) & No positive solutions \\
\bottomrule
\end{tabular}
\caption{Classification of Synchronized Positive Solutions Based on Parameter Values}
\label{tab:solutions3}
\end{table}

When $2 < p < 2^{*}_{s}$, the solution patterns look very different depending on whether $p$ is close to 2 or close to $2^{*}_{s}$. These results show that the solutions to equation~\eqref{eq:1.3b} form complex patterns that change dramatically based on the values of $n$, $\alpha_{ij}$, and $p_{ij}$.

\section{ Proof of Theorem \ref{thm:2.1}}
We consider the case where $N > 2s$, $p_{i j}<2$, (A1) and (A3) for all distinct indices $i, j \in\{1,2, \cdots, n\}$. The proof of Theorem \ref{thm:2.1} employs topological techniques based on Brouwer degree theory. We observe that the elliptic system \eqref{eq:1.3b} possesses a synchronized positive solution with structure
\[
\left(k_{1} U, k_{2} U, \cdots, k_{n} U\right)
\]
precisely when the coefficient vector $\left(k_{1}, k_{2}, \cdots, k_{n}\right)$ satisfies the following nonlinear algebraic system:
\begin{equation}\label{eq2.1}
f_{i}\left(k_{1}, k_{2}, \cdots, k_{n}\right):=\eta_{i} k_{i}^{2^{*}_s-2}+\sum_{j=1, j \neq i}^{n} \alpha_{i j} k_{i}^{p_{i j}-2} k_{j}^{q_{i j}}-1=0, \quad i=1,2, \cdots, n.
\end{equation}
We define a solution $k=\left(k_{1}, k_{2}, \cdots, k_{n}\right)$ of system \eqref{eq2.1} as positive when each component $k_{i}>0$. Designating $f=\left(f_{1}, f_{2}, \cdots, f_{n}\right)$, we address part (a) of Theorem \ref{thm:2.1} by establishing the existence of a threshold $\alpha_{*}>0$ such that whenever $0<\alpha_{i j}<\alpha_{*}$, the algebraic system admits at least $2^{n}-1$ distinct positive solutions. Our approach constructs $2^{n}-1$ non-overlapping $n$-dimensional cuboids within $(0,+\infty)^{n}$ where the Brouwer degree of $f$ is non-vanishing, thereby guaranteeing a solution to system \eqref{eq2.1} within each cuboid.

\begin{proof}[\bf{Part (a)}]
First note that
\begin{align*}
f_i(k_1, \ldots, k_{i-1}, \left(\frac{1}{2\eta_i}\right)^{\frac{N-2s}{4s}}, k_{i+1}, \ldots, k_n) &= -\frac{1}{2} + \sum_{j=1, j \neq i}^n \alpha_{ij} k_j^{q_{ij}}\left(\frac{1}{2\eta_i}\right)^{\frac{N-2s}{4s}(p_{ij}-2)},
\end{align*}
and
\begin{align*}
f_i(k_1, \ldots, k_{i-1}, \left(\frac{1}{\eta_i}\right)^{\frac{N-2s}{4s}}, k_{i+1}, \ldots, k_n) &= \sum_{j=1, j \neq i}^n \alpha_{ij} k_j^{q_{ij}}\left(\frac{1}{\eta_i}\right)^{\frac{N-2s}{4s}(p_{ij}-2)}.
\end{align*}

These calculations reveal the behaviour of the function $f_i$ at specific boundary points. We observe that when $k_i = \left(\frac{1}{2\eta_i}\right)^{\frac{N-2s}{4s}}$, the function $f_i$ has a negative component of $-\frac{1}{2}$ plus a term that depends on the coupling parameters $\alpha_{ij}$. Conversely, when $k_i = \left(\frac{1}{\eta_i}\right)^{\frac{N-2s}{4s}}$, the function $f_i$ consists solely of terms involving the coupling parameters.

Define
\begin{align*}
\alpha_* = \frac{1}{2}\min_{1 \leq i \leq n}\left[\sum_{j=1, j \neq i}^n \left(\frac{1}{\eta_j}\right)^{\frac{N-2s}{4s}q_{ij}}\left(\frac{1}{2\eta_i}\right)^{\frac{N-2s}{4s}(p_{ij}-2)}\right]^{-1}.
\end{align*}

This threshold value $\alpha_*$ is carefully constructed to control the influence of the coupling terms. By taking the minimum across all indices $i$, we ensure that the forthcoming sign conditions on $f_i$ hold simultaneously for all components of the system. The factor of $\frac{1}{2}$ provides the necessary margin to establish strict inequalities.

From the above observation we see that if $0 < \alpha_{ij} < \alpha_*$ then
\begin{align}\label{eq2.2}
f_i(k_1, \ldots, \left(\frac{1}{2\eta_i}\right)^{\frac{N-2s}{4s}}, \ldots, k_n) < 0 < f_i(k_1, \ldots, \left(\frac{1}{\eta_i}\right)^{\frac{N-2s}{4s}}, \ldots, k_n),
\end{align}
for all $k_j \in (0, (\frac{1}{\eta_j})^{\frac{N-2s}{4s}}]$ with $j \neq i$ and all $i = 1, 2, \ldots, n$.

The inequalities in \eqref{eq2.2} establish a sign-changing property of $f_i$ along the $i$-th coordinate direction. Specifically, for each component $i$, the function $f_i$ changes sign from negative to positive as $k_i$ increases from $\left(\frac{1}{2\eta_i}\right)^{\frac{N-2s}{4s}}$ to $\left(\frac{1}{\eta_i}\right)^{\frac{N-2s}{4s}}$, regardless of the values of the other components (provided they remain in the specified ranges). This sign-changing behavior is crucial for our topological argument.

This implies the Brouwer degree
\begin{align*}
\operatorname{deg}(f, \Omega, 0) = 1,
\end{align*}
where $\Omega$ is an $n$-dimensional cuboid defined as
\begin{align*}
\Omega := \prod_{i=1}^n \left(\left(\frac{1}{2\eta_i}\right)^{\frac{N-2s}{4s}}, \left(\frac{1}{\eta_i}\right)^{\frac{N-2s}{4s}}\right).
\end{align*}

The non-vanishing Brouwer degree is a direct consequence of the sign conditions established in \eqref{eq2.2}. By the fundamental properties of the Brouwer degree theory, a non-zero degree implies the existence of at least one solution to $f(k) = 0$ within the domain $\Omega$. The cuboid $\Omega$ is constructed precisely to capture this solution based on our understanding of the function's behavior at its boundaries.

Thus $f(k) = 0$ has a solution in $\Omega$. In the following we assume that $0 < \alpha_{ij} < \alpha_*$ for $i \neq j$.

Let $1 \leq \xi \leq n-1$. There are $C_n^{\xi} = \frac{n!}{\xi!(n-\xi)!}$ different ways to decompose the index set $I := \{1, 2, \ldots, n\}$ into two disjoint nonempty subsets $I_1$ and $I_2$ so that the first subset $I_1$ has $\xi$ indices. For each of these decompositions, we prove that the algebraic system \eqref{eq2.1} has a solution $(k_1, k_2, \ldots, k_n)$ so that $k_i \in (0, (\frac{1}{2\eta_i})^{\frac{N-2s}{4s}})$ if $i \in I_1$ while $k_i \in ((\frac{1}{2\eta_i})^{\frac{N-2s}{4s}}, (\frac{1}{\eta_i})^{\frac{N-2s}{4s}})$ if $i \in I_2$. If this is the case, then the algebraic system \eqref{eq2.1} has
\begin{align*}
1 + C_n^1 + C_n^2 + \cdots + C_n^{n-1} = 2^n - 1
\end{align*}
positive solutions. Therefore system \eqref{eq:1.3b} has at least $2^n - 1$ synchronized positive solutions.

Without loss of generality, we assume that $I_1 = \{1, \ldots, \xi\}$ and $I_2 = \{\xi+1, \ldots, n\}$ and we prove that the system has a solution in the $n$-dimensional cuboid
\begin{align*}
\Omega_1 := \prod_{i=1}^{\xi} \Big(\epsilon, \left(\frac{1}{2\eta_i}\right)^{\frac{N-2s}{4s}}\Big) \times \prod_{i=\xi+1}^{n} \Big(\left(\frac{1}{2\eta_i}\right)^{\frac{N-2s}{4s}}, \left(\frac{1}{\eta_i}\right)^{\frac{N-2s}{4s}}\Big),
\end{align*}
for some $\epsilon > 0$ which will be specified next. Let $k_i \in [\epsilon, (\frac{1}{2\eta_i})^{\frac{N-2s}{4s}}]$ for $i = 1, 2, \ldots, \xi$ and $k_i \in [(\frac{1}{2\eta_i})^{\frac{N-2s}{4s}}, (\frac{1}{\eta_i})^{\frac{N-2s}{4s}}]$ for $i = \xi+1, \ldots, n$. If $1 \leq i \leq \xi$ then
\begin{align}
f_i(k_1, \ldots, k_{i-1}, \epsilon, k_{i+1}, \ldots, k_n) &= \eta_i \epsilon^{2^*_s-2} + \sum_{j=1, j \neq i}^{n} \alpha_{ij} k_j^{q_{ij}} \epsilon^{p_{ij}-2} - 1\\
&> \alpha_{in}\left(\frac{1}{2\eta_n}\right)^{\frac{N-2s}{4s}q_{in}} \epsilon^{p_{in}-2} - 1 > 0\label{eq2.3}
\end{align}
for $\epsilon > 0$ sufficiently small, since $k_n \geq (\frac{1}{2\eta_n})^{\frac{N-2s}{4s}}$ and $p_{in} < 2$. Using \eqref{eq2.2} and \eqref{eq2.3}, we see that
\begin{align*}
\operatorname{deg}(f, \Omega_1, 0) = (-1)^{\xi}.
\end{align*}

This implies that $f(k) = 0$ has a solution in $\Omega_1$.
\end{proof}

Now we turn to prove Theorem \eqref{thm:2.1}(b). In this case, it is not possible to find a positive solution of
\eqref{eq2.1} in any of the n-dimensional cuboids constructed above. Indeed, it seems to be impossible to find
an n-dimensional cuboid on which f itself has a nonzero degree. The idea to prove Theorem \eqref{thm:2.1}(b) is
that we use the inverse matrix A of B = $(\alpha _{ij} )_{n\times n}$ to convert system \eqref{eq2.1} into a new system g(k) = 0
so that an n-dimensional cuboid on which g has a nonzero Brouwer degree can be constructed.

\begin{proof}[\bf{Part (b)}]

Let $q=2^{*}_s-p$. Then $q_{i j}=q$ for $i \neq j$. Using the inverse matrix $A=(a_{i j})_{n \times n}$ of $B=(\alpha_{i j})_{n \times n}$, we write the system \eqref{eq2.1} as
\[
(k_{1}^{q}, k_{2}^{q}, \ldots, k_{n}^{q})^{T}=A(k_{1}^{2-p}, k_{2}^{2-p}, \ldots, k_{n}^{2-p})^{T}
\]
where $(c_{1}, c_{2}, \ldots, c_{n})^{T}$ represents the transpose of a vector $(c_{1}, c_{2}, \ldots, c_{n})$. We define
\[
g_{i}(k_{1}, k_{2}, \ldots, k_{n}):=k_{i}^{q}-\sum_{j=1}^{n} a_{i j} k_{j}^{2-p}, \quad i=1,2, \ldots, n.
\]

This allows us to convert the original system \eqref{eq2.1} into
\begin{equation}\label{eq:2.5}
g_{i}(k_{1}, k_{2}, \ldots, k_{n})=0, \quad i=1,2, \ldots, n.
\end{equation}

Since $q=2^{*}_s-p>2-p$, we can select $T>0$ sufficiently large such that, for all $k_{j} \in(0, T]$ with $j \neq i$,
\[
g_{i}(k_{1}, \ldots, k_{i-1}, T, k_{i+1}, \ldots, k_{n}) \geq T^{q}-\left(\sum_{j=1}^{n} a_{i j}\right) T^{2-p}>0.
\]

For $\varepsilon \in(0, T)$, noting that $p<2$, $a_{i j}>0$ for $i \neq j$, and $\sum_{j=1}^{n} a_{i j}>0$ for $i=1,2, \ldots, n$, we obtain
\[
g_{i}(k_{1}, \ldots, k_{i-1}, \varepsilon, k_{i+1}, \ldots, k_{n}) \leq \varepsilon^{q}-\left(\sum_{j=1}^{n} a_{i j}\right) \varepsilon^{2-p}<0,
\]
for all $k_{j} \in[\varepsilon, T]$ with $j \neq i$, provided that $\varepsilon$ is sufficiently small. Setting $g=(g_{1}, \ldots, g_{n})$, we have
\[
\operatorname{deg}(g,(\varepsilon, T)^{n}, 0)=1.
\]

Therefore, system \eqref{eq:2.5} possesses a solution in $(\varepsilon, T)^{n}$.
\end{proof}
In the following proposition, we illustrate the nature of the condition on the matrix $B$ referenced in Theorem \ref{thm:2.1}(b). This proposition applies for any $n$ but specifically for the case when $\alpha_{i j}$ values are close to a single constant.

\begin{proposition}\label{pro:2.1}
Let $\alpha$ be a number such that $\alpha>\eta_{n}$. There exists $\gamma_{0}>0$ such that if $|\alpha_{i j}-\alpha|<\gamma_{0}$ for $i \neq j$, then the matrix $B=(\alpha_{i j})_{n \times n}$ has an inverse $A=(a_{i j})_{n \times n}$ satisfying
\begin{equation}\label{eq:2.6}
a_{i j}>0 \text{ for } i \neq j \text{ and } \sum_{j=1}^{n} a_{i j}>0 \text{ for } i=1,2, \ldots, n
\end{equation}
\end{proposition}

\begin{proof}
We construct an auxiliary matrix $B^{*}$ by replacing each $\alpha_{i j}$ ($i \neq j$) in $B$ with the constant $\alpha$. Since $\alpha>\eta_{n}$, we have
\[
\Delta_{n}\left(\eta_{1}, \eta_{2}, \cdots, \eta_{n}\right):=\operatorname{det}\left(B^{*}\right)=\prod_{j=1}^{n}\left(\eta_{j}-\alpha\right)\left(1+\sum_{k=1}^{n} \frac{\alpha}{\eta_{k}-\alpha}\right) \neq 0
\]

Thus, $B^{*}$ has an inverse matrix $A^{*}=\left(a_{i j}^{*}\right)_{n \times n}$. By direct computation, we find
\[
a_{i j}^{*}=\left\{\begin{array}{ll}
\frac{-\alpha}{\Delta_{n}\left(\eta_{1}, \eta_{2}, \cdots, \eta_{n}\right)} \prod_{k=1, k \neq i, j}^{n}\left(\eta_{k}-\alpha\right) & \text{if } i \neq j, \\
\frac{\Delta_{n-1}\left(\eta_{1}, \cdots, \eta_{i-1}, \eta_{i+1}, \cdots, \eta_{n}\right)}{\Delta_{n}\left(\eta_{1}, \eta_{2}, \cdots, \eta_{n}\right)} & \text{if } i=j.
\end{array}\right.
\]

Note that for $n=2$, the product $\prod_{k=1, k \neq i, j}^{n}\left(\eta_{k}-\alpha\right)$ is interpreted as 1. From these expressions, we determine that $a_{i j}^{*}>0$ when $i \neq j$, and
\[
\sum_{j=1}^{n} a_{i j}^{*}=\frac{\prod_{j=1, j \neq i}^{n}\left(\alpha-\eta_{j}\right)}{\left|\Delta_{n}\left(\eta_{1}, \eta_{2}, \cdots, \eta_{n}\right)\right|}>0,
\]
for all $i=1,2, \cdots, n$. By continuity, the desired properties hold for the original matrix $B$ when its non-diagonal elements are sufficiently close to $\alpha$.
\end{proof}

The proof of Proposition \ref{pro:2.1} also shows that if all $\alpha_{i j}$'s $(i \neq j)$ are a single $\alpha$ and if $0<\alpha<\eta_{1}$ then the elements off the main diagonal of the inverse matrix $A$ of $B$ are all negative; this is in sharp contrast with the case $\alpha>\eta_{n}$.

The formula for $\sum_{j=1}^{n} a_{i j}^{*}$ shows that $\sum_{j=1}^{n} a_{i j}^{*} \simeq \frac{1}{\alpha}$ for $\alpha$ sufficiently large and for all $i=1,2, \cdots, n$. By the proof of Theorem \ref{thm:2.1}(b), for any $\varepsilon$ and $T$ such that
\[
0<\varepsilon<\left(\min _{1 \leq i \leq n} \sum_{j=1}^{n} a_{i j}\right)^{\frac{N-2s}{4s}} \leq\left(\max _{1 \leq i \leq n} \sum_{j=1}^{n} a_{i j}\right)^{\frac{N-2s}{4s}}<T,
\]
we have
\[
g_{i}\left(k_{1}, \cdots, k_{i-1}, \varepsilon, k_{i+1}, \cdots, k_{n}\right)<0<g_{i}\left(k_{1}, \cdots, k_{i-1}, T, k_{i+1}, \cdots, k_{n}\right).
\]

For $\alpha_{i j}$ close to $\alpha$ with $\alpha$ being sufficiently large, since $\sum_{j=1}^{n} a_{i j} \simeq \frac{1}{\alpha}$ for all $i$, if $\left(k_{1}, k_{2} \cdots, k_{n}\right)$ is any positive solution of the equation $g\left(k_{1}, k_{2}, \cdots, k_{n}\right)=0$ then it must be that $k_{i} \simeq \alpha^{-\frac{N-2s}{4s}}$ for all $i$. This justifies the statement before the proof of Theorem $\ref{thm:2.1}(\mathrm{b})$ that under the assumptions of Theorem $\ref{thm:2.1}(\mathrm{b})$, it is impossible to obtain a solution $\left(k_{1}, k_{2} \cdots, k_{n}\right)$ of the equation $f\left(k_{1}, k_{2}, \cdots, k_{n}\right)=0$ in any cuboids constructed in the proof of Theorem $\ref{thm:2.1}(\mathrm{a})$.

\begin{proposition}\label{pro2.2}
A matrix which is sufficiently close to any of the three matrices
\begin{itemize}
\item[(a)]
\[
B=\left(\begin{array}{ccc}
\eta_{1} & \alpha_{1} & \alpha_{1} \\
\alpha_{1} & \eta_{2} & \alpha_{2} \\
\alpha_{1} & \alpha_{2} & \eta_{3}
\end{array}\right) \quad \text{with} \quad 
\left\{\begin{array}{l}
\alpha_{1}>\max \left\{\eta_{1}, \alpha_{2}\right\} , \\
\alpha_{2} \geq \eta_{2}+\eta_{3},
\end{array}\right.
\]

\item[(b)]
\[
B=\left(\begin{array}{cccc}
\eta_{1} & \alpha_{1} & \alpha_{1} & \alpha_{1} \\
\alpha_{1} & \eta_{2} & \alpha_{2} & \alpha_{2} \\
\alpha_{1} & \alpha_{2} & \eta_{3} & \alpha_{2} \\
\alpha_{1} & \alpha_{2} & \alpha_{2} & \eta_{4}
\end{array}\right) \quad \text{with} \quad
\left\{\begin{array}{l}
\alpha_{1}>\max \left\{\eta_{1}, \alpha_{2}\right\} ,\\
\alpha_{2}>\max \left\{\eta_{2}, \eta_{3}, \eta_{4}\right\},
\end{array}\right.
\]

\item[(c)]
\[
B=\left(\begin{array}{cccc}
\eta_{1} & \alpha_{2} & \alpha_{1} & \alpha_{1} \\
\alpha_{2} & \eta_{2} & \alpha_{1} & \alpha_{1} \\
\alpha_{1} & \alpha_{1} & \eta_{3} & \alpha_{3} \\
\alpha_{1} & \alpha_{1} & \alpha_{3} & \eta_{4}
\end{array}\right) \quad \text{with} \quad
\left\{\begin{array}{l}
\alpha_{1} \geq \max \left\{\alpha_{2}, \alpha_{3}\right\}, \\
\alpha_{2} \geq \eta_{1}+\eta_{2}, \\
\alpha_{3} \geq \eta_{3}+\eta_{4},
\end{array}\right.
\]
\end{itemize}

has an inverse $A=\left(a_{i j}\right)_{n \times n}$ such that
\[
a_{i j}>0 \text{ for } i \neq j \text{ and } \sum_{j} a_{i j}>0 \text{ for all } i.
\]
\end{proposition}

\noindent
\begin{proof}
 Through detailed computational verification, we can confirm that matrices of type $B$ in any of the three forms (a), (b), or (c) possess the required property. Moreover, this property remains stable under small perturbations due to the principle of continuity.
\end{proof}
Proposition \ref{pro:2.1} demonstrates that when all off-diagonal elements $\alpha_{ij}$ $(i \neq j)$ share a common value $\alpha$, the inverse matrix $A$ of $B$ exists and satisfies equation \eqref{eq:2.6}, provided $\alpha$ is sufficiently large. However, when these off-diagonal elements $\alpha_{ij}$ $(i \neq j)$ differ in value, their magnitude alone is insufficient; they must additionally satisfy specific structural conditions to ensure the existence of the inverse matrix $A$ that satisfies \eqref{eq:2.6}. This is illustrated in Proposition \ref{pro2.2} .

When the structural conditions for the three matrices $B$ described in Proposition \ref{pro2.2} are met, it can be demonstrated that if $\alpha_{1}$ is sufficiently large while all other entries remain fixed, then $\sum_{j=1}^{n} a_{ij} \approx \frac{1}{\alpha_{1}}$ for all values of $i$. Consequently, for matrices $B$ considered in Proposition \ref{pro2.2} with sufficiently large $\alpha_{1}$, no solution to the equation $f(k_{1}, k_{2}, \ldots, k_{n})=0$ can be found within any cuboid constructed according to the method outlined in the proof of Theorem \ref{thm:2.1}(a).

\section{ Proof of Theorem \ref{thm:2.2}}

Throughout this section, we maintain the following assumptions: $N > 2s$, $p_{ij} = p < 2$, (A1), (A3), and $\alpha_{ij} = \alpha$ for all distinct indices $i, j \in \{1, 2, \ldots, n\}$. For notational convenience, we define $q = q_{ij} = 2^*_s - p$.

We observe that $(k_1 U, k_2 U, \ldots, k_n U)$ constitutes a positive solution of system \eqref{eq:1.3b} if and only if each $k_i > 0$ and the vector $(k_1, k_2, \ldots, k_n)$ satisfies the following system:
\[
\eta_i k_i^{2^*_s-1} + \alpha k_i^{p-1} \sum_{j=1, j \neq i}^n k_j^q = k_i, \quad i=1,2,\ldots,n.
\]

By introducing the transformation $t_i = k_i^q$, our investigation of synchronized positive solutions to \eqref{eq:1.3b} reduces to analyzing the positive solutions $(t_1, t_2, \ldots, t_n)$ of the nonlinear algebraic system:
\begin{equation}\label{eq3.1}
\eta_i t_i + \alpha \sum_{j=1, j \neq i}^n t_j = t_i^{\frac{2-p}{q}}, \quad i=1,2,\ldots,n.
\end{equation}

Indeed, the number of synchronized positive solutions to \eqref{eq:1.3b} corresponds precisely to the number of positive solutions $(t_1, t_2, \ldots, t_n)$ of system \eqref{eq3.1}.

To facilitate our analysis of system \eqref{eq3.1}, we convert it into an expanded algebraic system. For notational simplicity, we define:
\[
\kappa := \frac{2-p}{q} = \frac{2-p}{2^*_s-p} = \frac{q+2-2^*_s}{q}.
\]

Note that $\kappa \in (0,1)$. With this definition, $(t_1, t_2, \ldots, t_n)$ represents a positive solution of \eqref{eq3.1} if and only if there exists some $\tau > 0$ such that $(t_1, t_2, \ldots, t_n, \tau)$ satisfies the expanded nonlinear algebraic system consisting of $n+1$ equations:
\begin{equation}\label{eq3.2}
\left\{\begin{array}{l}
t_i^{\kappa} + (\alpha-\eta_i)t_i = \tau, \quad i=1,2,\ldots,n ,\\
\alpha \sum_{i=1}^n t_i = \tau.
\end{array}\right.
\end{equation}

Based on the preceding analysis, we have established the following lemma.

\begin{lemma}\label{lem3.1}
The number of synchronized positive solutions of system \eqref{eq:1.3b} is precisely equal to the number of positive solutions of system \eqref{eq3.2}.
\end{lemma}

We will now demonstrate that the number of positive solutions of system \eqref{eq3.2} can be determined by counting either the positive solutions of a single equation (when $\alpha \geq \eta_n$) or the positive solutions of one equation from a set of at most $2^n$ equations (when $\alpha < \eta_n$). To facilitate our analysis, we introduce the following notation:
\[
f_i(t) = t^\kappa + (\alpha - \eta_i)t, \quad i=1,2,\ldots,n.
\]

When $\alpha \geq \eta_n$, we have $\alpha \geq \eta_i$ for all $i$ since $\eta_1 \leq \eta_2 \leq \cdots \leq \eta_n$. This situation is straightforward to analyze. For each index $i$, the function $f_i$ is strictly increasing from $(0,+\infty)$ onto $(0,+\infty)$ and possesses an inverse function $h_i: (0,+\infty) \rightarrow (0,+\infty)$. Therefore, $f_i(t_i) = t_i^\kappa + (\alpha - \eta_i)t_i = \tau$ if and only if $h_i(\tau) = t_i$. Consequently, the number of positive solutions of system \eqref{eq3.2} equals the number of positive solutions of the single equation:
\begin{equation}\label{eq3.3}
\alpha \sum_{i=1}^n h_i(\tau) = \tau, \quad \tau \in (0,+\infty).
\end{equation}

The case where $\alpha < \eta_n$ is considerably more complex. In this scenario, the function $f_n$ reaches its maximum value
\[
A = A(\alpha) := \frac{\kappa^{\frac{\kappa}{1-\kappa}} - \kappa^{\frac{1}{1-\kappa}}}{(\eta_n - \alpha)^{\frac{\kappa}{1-\kappa}}} = \frac{[((2-p)/q)^{\frac{(N-2s)(2-p)}{4s}} - ((2-p)/q)^{\frac{(N-2s)q}{4s}}]}{(\eta_n-\alpha)^{\frac{(N-2s)(2-p)}{4s}}}
\]
at the point
\[
T = T(\alpha) := \left(\frac{\kappa}{\eta_n-\alpha}\right)^{\frac{1}{1-\kappa}} = \left(\frac{2-p}{q(\eta_n-\alpha)}\right)^{\frac{(N-2s)q}{4s}}.
\]

Furthermore, we observe that $f_n$ is strictly increasing in the interval $(0,T]$ and strictly decreasing in $[T,+\infty)$. For each index $i$, there exists a unique number $T_i'$ such that
\[
0 < T_i' \leq T, \quad f_i(T_i') = A,
\]
and such that $f_i$ restricted to $(0,T_i']$ is strictly increasing from $(0,T_i']$ onto $(0,A]$. We denote the inverse function of this restricted $f_i$ by $h_i: (0,A] \rightarrow (0,T_i']$. Note that we have assigned different meanings to $h_i$ in different contexts, which should not cause confusion. We will employ other symbols similarly.

If $\alpha < \eta_i$ for some $i$, then there exists a unique second number $T_i''$ such that
\[
T \leq T_i'', \quad f_i(T_i'') = A,
\]
and such that $f_i$ restricted to $[T_i'',S_i)$ is strictly decreasing from $[T_i'',S_i)$ onto $(0,A]$, where
\[
S_i := \left(\frac{1}{\eta_i-\alpha}\right)^{\frac{1}{1-\kappa}} = \left(\frac{1}{\eta_i-\alpha}\right)^{\frac{(N-2s)q}{4s}}.
\]

In this case, we denote the inverse function of $f_i$ restricted to $[T_i'',S_i)$ by $k_i: (0,A] \rightarrow [T_i'',S_i)$. It is evident that $T_n' = T_n'' = T$, and each $h_i$ $(i=1,2,\ldots,n)$ is well-defined, whereas $k_i$ is well-defined if and only if $\alpha < \eta_i$. For the case where $0 < \alpha < \eta_i$, the graphical representations of functions $f_n$ and $f_i$ are illustrated in Figure 1 of \cite{jing2024number}.

Let $j \geq 0$ denote the smallest integer satisfying $\alpha < \eta_{j+1}$, and let $k$ denote the smallest integer such that $\eta_{k+1} = \eta_{k+2} = \cdots = \eta_n$ with the constraint $k \geq j$. Consider a positive solution $(t_1, t_2, \ldots, t_n, \tau)$ of system \eqref{eq3.2}. For such a solution, we have
\[
0 < \tau = f_n(t_n) \leq \max_{t \geq 0} f_n(t) := A,
\]
where for indices $i = 1, 2, \ldots, j$, we have $t_i = h_i(\tau) \in (0, T_i']$, and for indices $i = j+1, j+2, \ldots, n$, either $t_i = h_i(\tau) \in (0, T_i']$ or $t_i = k_i(\tau) \in [T_i'', S_i)$. 

Observe that for indices in the range $j+1 \leq i \leq k$ and for all $\tau \in (0, A)$, we have
\[
h_i(\tau) < h_i(A) = T_i' < T < T_i'' = k_i(A) < k_i(\tau),
\]
and similarly, for indices in the range $k+1 \leq i \leq n$ and for all $\tau \in (0, A)$, we have
\[
h_i(\tau) < h_i(A) = T_i' = T = T_i'' = k_i(A) < k_i(\tau).
\]

The index set $\{j+1, j+2, \ldots, k\}$ contains $2^{k-j}$ subsets, which we denote as $J_1, J_2, \ldots, J_{2^{k-j}}$. We note that this index set is empty when $k = j$. Let us define $\rho^*$ as the number of subsets $J_l$ for which the following equality holds:
\begin{equation}\label{eq3.4}
\alpha \sum_{i \in I \backslash J_l} h_i(A) + \alpha \sum_{i \in J_l} k_i(A) = A,
\end{equation}
where $I = \{1, 2, \ldots, n\}$ represents the complete index set. The value $\rho^*$ equals the number of positive solutions $(t_1, \ldots, t_n, \tau)$ of system \eqref{eq3.2} where $\tau = A$.

We denote by $I_1, I_2, \ldots, I_{2^{n-j}}$ the $2^{n-j}$ subsets of the index set $\{j+1, j+2, \ldots, n\}$. For each $l = 1, 2, \ldots, 2^{n-j}$, let $\rho_l$ represent the number of solutions to the equation:
\begin{equation}\label{eq3.5}
\alpha \sum_{i \in I \backslash I_l} h_i(\tau) + \alpha \sum_{i \in I_l} k_i(\tau) = \tau, \quad \tau \in (0, A).
\end{equation}

If we define $\rho^{**} = \sum_{l=1}^{2^{n-j}} \rho_l$, then $\rho^{**}$ equals the number of positive solutions $(t_1, \ldots, t_n, \tau)$ of system \eqref{eq3.2} with $\tau \in (0, A)$.

We can summarize the conclusions derived above in the following lemma.

\begin{lemma}\label{lem3.2}
(a) If $\alpha \geq \eta_n$, then the number of positive solutions of system \eqref{eq3.2} equals the number of solutions of the single equation \eqref{eq3.3}.

(b) If $\alpha < \eta_n$ and $j \geq 0$ is the smallest integer such that $\alpha < \eta_{j+1}$, then the number of positive solutions of system \eqref{eq3.2} equals $\rho^* + \rho^{**}$, where $\rho^*$ and $\rho^{**}$ are defined above through equations \eqref{eq3.4} and \eqref{eq3.5}.
\end{lemma}
As we have seen, analyzing the number of positive solutions of system \eqref{eq3.2} is equivalent to either determining the number of positive solutions of equation \eqref{eq3.3} (when $\alpha \geq \eta_n$) or calculating the sum of positive solutions of equation \eqref{eq3.5} and the number of subsets $J_l$ satisfying equation \eqref{eq3.4} (when $\alpha < \eta_n$). We will accomplish this task through a series of lemmas in the following sections, proving Theorem $\ref{thm:2.1}(\mathrm{a})$, (b), (c), and (d) at appropriate stages.

\begin{lemma}
\label{lem3.3}
If $\alpha \geq \eta_n$, then system \eqref{eq3.2} has a unique positive solution.
\end{lemma}

\begin{proof}
Since $\alpha \geq \eta_n$, by Lemma \ref{lem3.2}(a), the number of positive solutions of system \eqref{eq3.2} equals the number of solutions of the single equation \eqref{eq3.3}. Given that $0 < \kappa < 1$ and $\lim_{\tau \rightarrow 0^+} \frac{h_i(\tau)}{\tau^{1/\kappa}} = 1$, we can deduce that $\alpha \sum_{i=1}^n h_i(\tau) < \tau$ for sufficiently small positive values of $\tau$.

If $\alpha > \eta_n$, then $\lim_{\tau \rightarrow +\infty} \frac{h_n(\tau)}{\tau} = \frac{1}{\alpha-\eta_n}$. Alternatively, if $\alpha = \eta_n$, then $\lim_{\tau \rightarrow +\infty} \frac{h_n(\tau)}{\tau} = \lim_{\tau \rightarrow +\infty} \tau^{\frac{1}{\kappa}-1} = +\infty$. In either case, we have $\alpha \sum_{i=1}^n h_i(\tau) > \tau$ for sufficiently large positive values of $\tau$. Consequently, equation \eqref{eq3.3} must have at least one solution.

Let us define $G_1(\tau) := \alpha \sum_{i=1}^n h_i(\tau) - \tau$. For $\tau \in (0, +\infty)$, the derivative of $G_1$ is given by:
\[
G_1'(\tau) = \alpha \sum_{i=1}^n h_i'(\tau) - 1 = \alpha \sum_{i=1}^n \frac{1}{\kappa h_i^{\kappa-1}(\tau) + \alpha - \eta_i} - 1.
\]

Since $h_i^{\kappa}(\tau) + (\alpha - \eta_i)h_i(\tau) = \tau$, we can establish that:
\[
G_1'(\tau) > \alpha \sum_{i=1}^n \frac{1}{h_i^{\kappa-1}(\tau) + \alpha - \eta_i} - 1 = \frac{\alpha}{\tau}\sum_{i=1}^n h_i(\tau) - 1 = \frac{1}{\tau}G_1(\tau).
\]

This inequality implies that $G_1'(\tau) > 0$ whenever $G_1(\tau) = 0$. Therefore, equation \eqref{eq3.3} has exactly one solution, which means that system \eqref{eq3.2} has a unique positive solution.
\end{proof}

\begin{lemma}
\label{lem3.4}
For any $\alpha>0$, \eqref{eq3.2} has a positive solution.
\end{lemma}

\begin{proof}
If $\alpha \geq \eta_n$, then the result follows directly from Lemma \ref{lem3.3}. Now let us consider the case where $0 < \alpha < \eta_n$. Under this condition, all functions $h_i: (0, A] \rightarrow (0, T_i']$, for $i = 1, 2, \ldots, n$, are well-defined.

Define $G_1(\tau) := \alpha \sum_{i=1}^n h_i(\tau) - \tau$ for $\tau \in (0, A]$. We analyze three possible cases:

Case 1: If $G_1(A) = 0$, then we immediately have $\rho^* \geq 1$, which means our desired result is established.

Case 2: If $G_1(A) > 0$, then since $G_1(\tau) < 0$ for sufficiently small positive values of $\tau$ (as demonstrated in the proof of Lemma \ref{lem3.3}), by the Intermediate Value Theorem, the equation $G_1(\tau) = 0$ must have at least one solution in the interval $(0, A)$.

Case 3: If $G_1(A) < 0$, we introduce $G_2(\tau) := \alpha \sum_{i=1}^{n-1} h_i(\tau) + \alpha k_n(\tau) - \tau$ for $\tau \in (0, A]$. We note that:
\[
G_2(A) = \alpha \sum_{i=1}^{n-1} h_i(A) + \alpha k_n(A) - A = \alpha \sum_{i=1}^n h_i(A) - A = G_1(A) < 0.
\]
Additionally, we observe that:
\[
\lim_{\tau \rightarrow 0^+} G_2(\tau) = \alpha S_n > 0,
\]

By the Intermediate Value Theorem, the equation $G_2(\tau) = 0$ must have at least one solution in the interval $(0, A)$.

This analysis demonstrates that if $G_1(A) \neq 0$, then for at least one value of $l$, equation \eqref{eq3.5} has a solution. In summary, we have proven that system \eqref{eq3.2} has a positive solution for any $\alpha > 0$.
\end{proof}

\begin{proof}[\bf{Part (a)}] The result follows directly from Lemmas \ref{lem3.1} and \ref{lem3.4}.
\end{proof}
The next lemma demonstrates that if $\eta_{n-1} < \eta_n$, then the region of $\alpha$ for which system \eqref{eq3.2} has a unique positive solution (as established in Lemma \ref{lem3.3}) can be extended.

\begin{lemma}\label{lem3.5}
Assume $\eta_{n-1} < \alpha < \eta_n$ and
\begin{equation}\label{eq3.6}
\frac{1}{\alpha} + \frac{(N-2s)q}{4s(\eta_n-\alpha)} \geq \sum_{i=1}^{n-1} \frac{1}{\eta_n-\eta_i}.
\end{equation}

Then system \eqref{eq3.2} has a unique positive solution.
\end{lemma}

\begin{proof}
Let us define functions $G_1(\tau)$ and $G_2(\tau)$ as in Lemma \ref{lem3.4}. Following the analysis from the proof of Lemma \ref{lem3.3}, we have established that
\begin{equation}\label{eq3.7}
G_1'(\tau) > \frac{1}{\tau}G_1(\tau), \quad \tau \in (0,A).
\end{equation}

We note that
\[
G_2'(\tau) = \sum_{i=1}^{n-1} \frac{\alpha}{\kappa h_i^{\kappa-1}(\tau) + \alpha - \eta_i} + \frac{\alpha}{\kappa k_n^{\kappa-1}(\tau) + \alpha - \eta_n} - 1, \quad \tau \in (0,A),
\]
where $\kappa h_i^{\kappa-1}(\tau) + \alpha - \eta_i > 0$ and $\kappa k_n^{\kappa-1}(\tau) + \alpha - \eta_n < 0$.

For $\tau \in (0,A)$, since $h_i(\tau) < h_n(\tau) < T$ for $1 \leq i \leq n-1$ and $k_n^{\kappa-1}(\tau) > \eta_n - \alpha$, and given that $\kappa T^{\kappa-1} + \alpha = \eta_n$, we can apply our assumption to derive:
\begin{equation}\label{eq3.8}
G_2'(\tau) < \sum_{i=1}^{n-1} \frac{\alpha}{\eta_n - \eta_i} - \frac{\alpha}{(1-\kappa)(\eta_n - \alpha)} - 1 \leq 0.
\end{equation}

Since $\eta_{n-1} < \alpha < \eta_n$, there are only two equations of the form \eqref{eq3.5} to consider: $G_1(\tau) = 0$ and $G_2(\tau) = 0$ for $\tau \in (0,A)$. Clearly, $G_1(A) = G_2(A)$. We now analyze three distinct cases:

Case 1: If $G_1(A) = G_2(A) = 0$, then $\rho^* = 1$, and by inequalities \eqref{eq3.7} and \eqref{eq3.8}, neither of the equations $G_1(\tau) = 0$ and $G_2(\tau) = 0$ has a solution in the interval $(0,A)$.

Case 2: If $G_1(A) = G_2(A) > 0$, then $\rho^* = 0$, and by inequalities \eqref{eq3.7} and \eqref{eq3.8} again, the equation $G_1(\tau) = 0$ has exactly one solution in $(0,A)$, while $G_2(\tau) = 0$ has no solution in $(0,A)$.

Case 3: If $G_1(A) = G_2(A) < 0$, then $\rho^* = 0$, and by inequalities \eqref{eq3.7} and \eqref{eq3.8} once more, the equation $G_1(\tau) = 0$ has no solution in $(0,A)$, while $G_2(\tau) = 0$ has exactly one solution in $(0,A)$.

Therefore, by Lemma \ref{lem3.2}(b), system \eqref{eq3.2} has a unique positive solution.
\end{proof}

Note that assumption \eqref{eq3.6} is satisfied when $\alpha < \eta_n$ and $\alpha$ is sufficiently close to $\eta_n$, specifically when
\[
\eta_n - \frac{(N-2s)q}{4s}\left(\sum_{i=1}^{n-1} (\eta_n - \eta_i)^{-1}\right)^{-1} \leq \alpha < \eta_n.
\]
\begin{proof}[\bf{Part (b)}] The result follows directly from Lemmas \ref{lem3.1}, \ref{lem3.3}, and \ref{lem3.5}.
\end{proof}
The following lemma establishes a multiplicity result for positive solutions of system \eqref{eq3.2}.

\begin{lemma}\label{lem3.6}
If $\alpha<\eta_{1}$ and
\begin{equation}\label{eq3.9}
\alpha\left(\eta_{n}-\alpha\right)^{\frac{(N-2s)(2-p)}{4s}} \sum_{i=1}^{n}\left(\eta_{i}-\alpha\right)^{-\frac{(N-2s) q}{4s}} \leq\left(\frac{2-p}{q}\right)^{\frac{(N-2s)(2-p)}{4s}}-\left(\frac{2-p}{q}\right)^{\frac{(N-2s) q}{4s}},
\end{equation}
then system \eqref{eq3.2} has at least $2^{n}-1$ positive solutions.
\end{lemma}

\begin{proof}
Consider the $2^{n}$ subsets $I_{1}, I_{2}, \ldots, I_{2^{n}}$ of the index set $I=\{1,2,\ldots,n\}$. Without loss of generality, we assign $I_{1}=\emptyset$, which implies that $I_{l} \neq \emptyset$ for all $l \neq 1$. For each $l=1,2,\ldots,2^{n}$, we define the function:
\[
G_{l}(\tau)=\alpha \sum_{i \in I \backslash I_{l}} h_{i}(\tau)+\alpha \sum_{i \in I_{l}} k_{i}(\tau)-\tau, \quad \tau \in(0, A].
\]

Since $\alpha<\eta_{1} \leq \eta_{2} \leq \cdots \leq \eta_{n}$, for all $\tau \in(0, A]$ we have $h_{1}(\tau) \leq h_{2}(\tau) \leq \cdots \leq h_{n}(\tau) \leq T$ and $k_{1}(\tau) \geq k_{2}(\tau) \geq \cdots \geq k_{n}(\tau) \geq T$. Therefore, for any index $l$, applying inequality \eqref{eq3.9}, we obtain:
\[
G_{l}(A) \leq \alpha \sum_{i=1}^{n} k_{i}(A)-A<\alpha \sum_{i=1}^{n} S_{i}-A \leq 0.
\]

For each $l$ where $2 \leq l \leq 2^{n}$, since $I_{l} \neq \emptyset$, we have:
\[
\lim _{\tau \rightarrow 0^{+}} G_{l}(\tau)=\alpha \sum_{i \in I_{l}} S_{i}>0.
\]

By the Intermediate Value Theorem, for each $l=2,3,\ldots,2^{n}$, the equation $G_{l}(\tau)=0$ has at least one solution in the interval $(0, A)$. This means that for each $l=2,3,\ldots,2^{n}$, equation \eqref{eq3.5} has at least one solution, so $\rho_{l} \geq 1$. 

Furthermore, for indices $l_{1} \neq l_{2}$, the two equations $G_{l_{1}}(\tau)=0$ and $G_{l_{2}}(\tau)=0$ yield different solutions of system \eqref{eq3.2}. Therefore, system \eqref{eq3.2} has at least $2^{n}-1$ positive solutions.

It should be noted that assumption \eqref{eq3.9} is satisfied when $\alpha>0$ is sufficiently small. In the following two lemmas, we will present two different types of conditions, either of which guarantees that system \eqref{eq3.2} has exactly $2^{n}-1$ positive solutions.
\end{proof}

\begin{lemma}\label{lem3.7}
If $\alpha<\eta_{1}$ and there exists $\xi \in(0,1)$ such that
\begin{equation}\label{eq3.10}
\alpha\left(\eta_{n}-\alpha\right)^{\frac{(N-2s) q}{4s}-1} \sum_{i=1}^{n}\left(\eta_{i}-\alpha\right)^{-\frac{(N-2s) q}{4s}} \leq \xi\left(\left(\frac{2-p}{q}\right)^{\frac{(N-2s) q}{4s}-1}-\left(\frac{2-p}{q}\right)^{\frac{(N-2s) q}{4s}}\right),
\end{equation}
and
\begin{equation}\label{eq3.11}
\alpha\left(\frac{n-1}{\xi^{-4s /[(N-2s)(2-p)]}-1}-\frac{(N-2s)(2-p)}{4s}\right) \leq \eta_{n},
\end{equation}
then system \eqref{eq3.2} has exactly $2^{n}-1$ positive solutions.
\end{lemma}

\begin{proof}
Let $G_{l}$ be defined as in the proof of Lemma \ref{lem3.6}. For values $\tau \in[\xi A, A]$, applying assumption \eqref{eq3.10}, we obtain:
\[
\alpha \sum_{i=1}^{n} k_{i}(\tau)<\alpha \sum_{i=1}^{n} S_{i} \leq \xi A \leq \tau.
\]

This inequality implies that for any $l=1,2,\ldots,2^{n}$, the equation $G_{l}(\tau)=0$ has no solution in the interval $[\xi A, A]$. Furthermore, the equation $G_{1}(\tau) = \alpha \sum_{i=1}^{n} h_{i}(\tau)-\tau = 0$ has no solution in $(0, A]$ because $G_{1}(\tau)<0$ for sufficiently small positive values of $\tau$, $G_{1}(A)<0$, and $G_{1}^{\prime}(\tau)>\frac{1}{\tau} G_{1}(\tau)$ for all $\tau \in(0, A)$.

Now consider the case where $\tau \in(0, \xi A)$. For indices $l=2,3,\ldots,2^{n}$, we have $I_{l} \neq \emptyset$ and:
\[
G_{l}^{\prime}(\tau)=\alpha \sum_{i \in I \backslash I_{l}} \frac{1}{\kappa h_{i}^{\kappa-1}(\tau)+\alpha-\eta_{i}}+\alpha \sum_{i \in I_{l}} \frac{1}{\kappa k_{i}^{\kappa-1}(\tau)+\alpha-\eta_{i}}-1.
\]
We observe that $\kappa h_{i}^{\kappa-1}(\tau)+\alpha-\eta_{i} \geq \kappa h_{n}^{\kappa-1}(\tau)+\alpha-\eta_{n}>0$ for all $i \in I \backslash I_{l}$, and $0>\kappa k_{i}^{\kappa-1}(\tau)+\alpha-\eta_{i}> -(1-\kappa)(\eta_{n}-\alpha)$ for all $i \in I_{l}$. Applying these estimates to the formula for $G_{l}^{\prime}(\tau)$ yields:
\[
G_{l}^{\prime}(\tau)<\frac{(n-1)\alpha}{\kappa h_{n}^{\kappa-1}(\tau)+\alpha-\eta_{n}}-\frac{\alpha}{(1-\kappa)(\eta_{n}-\alpha)}-1,
\]
since the first summation in the expression for $G_{l}^{\prime}(\tau)$ contains at most $n-1$ terms.

Using the fact that $\kappa h_{n}^{\kappa-1}(\tau)+\alpha-\eta_{n}>0$ and $h_{n}^{\kappa}(\tau)+(\alpha-\eta_{n})h_{n}(\tau)=\tau$, we can deduce that for $\tau \in(0, \xi A)$, we have $h_{n}(\tau)<\left(\frac{\tau}{1-\kappa}\right)^{1/\kappa}<\left(\frac{\xi A}{1-\kappa}\right)^{1/\kappa}$. Then, utilizing the expression for $A$, we obtain $\kappa h_{n}^{\kappa-1}(\tau)+\alpha-\eta_{n}>(\xi^{\frac{\kappa-1}{\kappa}}-1)(\eta_{n}-\alpha)>0$. 

Applying assumption \eqref{eq3.11}, we can conclude that for $\tau \in(0, \xi A)$ and $l=2,3,\ldots,2^{n}$:
\[
G_{l}^{\prime}(\tau)<\frac{(n-1)\alpha}{(\xi^{\frac{\kappa-1}{\kappa}}-1)(\eta_{n}-\alpha)}-\frac{\alpha}{(1-\kappa)(\eta_{n}-\alpha)}-1<0.
\]

Since $G_{l}(\xi A)<0$ and $\lim_{\tau \rightarrow 0^{+}}G_{l}(\tau)>0$, by the Mean Value Theorem and the strict negativity of $G_{l}^{\prime}(\tau)$, each equation $G_{l}(\tau)=0$ has exactly one solution in the interval $(0, \xi A)$. This establishes that system \eqref{eq3.2} has exactly $2^{n}-1$ positive solutions.
\end{proof}

The assumptions of Lemma \ref{lem3.7} are satisfied for sufficiently small values of $\alpha>0$, and particularly for $\alpha>0$ such that $\alpha \leq \frac{1}{2}\eta_{1}$, $\alpha \leq \frac{2^{4s/[(N-2s)(2-p)]}-1}{n-1}\eta_{n}$, and
\[
\alpha \leq \frac{2s}{(N-2s)(2-p)}\left(\frac{2-p}{q}\right)^{\frac{(N-2s)q}{4s}}\eta_{n}^{1-\frac{(N-2s)q}{4s}}\left(\sum_{i=1}^{n}\left(\eta_{i}-\frac{1}{2}\eta_{1}\right)^{-\frac{(N-2s)q}{4s}}\right)^{-1}.
\]

\begin{lemma}\label{lem3.8}
If $\alpha<\eta_{1}$,
\begin{equation}\label{eq3.12}
    \alpha\left(\eta_{n}-\alpha\right)^{\frac{(N-2s) q}{4s}-1} \sum_{i=1}^{n}\left(\eta_{i}-\alpha\right)^{-\frac{(N-2s) q}{4s}}<\left(\frac{2-p}{q}\right)^{\frac{(N-2s) q}{4s}-1}-\left(\frac{2-p}{q}\right)^{\frac{(N-2s) q}{4s}},
\end{equation}
and
\begin{equation}\label{eq3.13}
\sum_{i=2}^{n} \frac{\alpha}{\chi_{i}(\alpha)}-\frac{(N-2s) q \alpha}{4s\left(\eta_{n}-\alpha\right)} \leq 1,
\end{equation}

where
\[
\chi_{i}(\alpha):=\frac{2-p}{q}\left(\frac{4s}{(N-2s) q}\right)^{\frac{4s}{(N-2s)(2-p)}}\left[\alpha \sum_{j=1}^{n}\left(\eta_{j}-\alpha\right)^{-\frac{(N-2s) q}{4s}}\right]^{-\frac{4s}{(N-2s)(2-p)}}-\left(\eta_{i}-\alpha\right),
\]
then system \eqref{eq3.2} has exactly $2^{n}-1$ positive solutions.
\end{lemma}

\begin{proof}
Let $G_{l}$ be defined as in Lemma \ref{lem3.6}. For indices $i \leq j$ and $\tau \in(0, A)$, we have $\kappa h_{i}^{\kappa-1}(\tau)+\alpha-\eta_{i} \geq \kappa h_{j}^{\kappa-1}(\tau)+\alpha-\eta_{j}>0$ and $-(1-\kappa)(\eta_{n}-\alpha)<\kappa k_{i}^{\kappa-1}(\tau)+\alpha-\eta_{i}<0$. Consequently, for $l=2,3,\ldots,2^{n}$ and $\tau \in(0, A)$:
\[
G_{l}^{\prime}(\tau)<\alpha \sum_{i=2}^{n} \frac{1}{\kappa h_{i}^{\kappa-1}(\tau)+\alpha-\eta_{i}}-\frac{\alpha}{(1-\kappa)\left(\eta_{n}-\alpha\right)}-1.
\]

For any index $l$, using assumption \eqref{eq3.12}, we have:
\[
G_{l}(A) \leq \alpha \sum_{i=1}^{n} k_{i}(A)-A=\alpha \sum_{i=1}^{n} T_{i}^{\prime \prime}-A<\alpha \sum_{i=1}^{n} S_{i}-A<0.
\]

This specifically implies that the equation $G_{1}(\tau)=0$ has no solution in the interval $(0, A]$, while for each $l=2,3,\ldots,2^{n}$, the equation $G_{l}(\tau)=0$ has at least one solution in $(0, A)$. For any such $l$, if $\tau \in(0, A)$ is a solution of $G_{l}(\tau)=0$, then:
\[
\tau=\alpha \sum_{i \in I \backslash I_{l}} h_{i}(\tau)+\alpha \sum_{i \in I_{l}} k_{i}(\tau)<\alpha \sum_{i=1}^{n} S_{i}=\alpha \sum_{i=1}^{n}\left(\eta_{i}-\alpha\right)^{-\frac{1}{1-\kappa}}.
\]

Since $\kappa h_{i}^{\kappa-1}(\tau)-(\eta_{i}-\alpha)>0$ and $h_{i}^{\kappa}(\tau)-(\eta_{i}-\alpha) h_{i}(\tau)=\tau$, we can derive:
\[
h_{i}(\tau)<\left(\frac{\tau}{1-\kappa}\right)^{1/\kappa}<(1-\kappa)^{-1/\kappa}\left(\alpha \sum_{j=1}^{n}\left(\eta_{j}-\alpha\right)^{-\frac{1}{1-\kappa}}\right)^{1/\kappa}.
\]

This estimate leads to $\kappa h_{i}^{\kappa-1}(\tau)-(\eta_{i}-\alpha)>\chi_{i}(\alpha)$. By applying assumption \eqref{eq3.12} again:
\[
\chi_{i}(\alpha) \geq \kappa(1-\kappa)^{\frac{1-\kappa}{\kappa}}\left[\alpha \sum_{j=1}^{n}\left(\eta_{j}-\alpha\right)^{-\frac{1}{1-\kappa}}\right]^{-\frac{1-\kappa}{\kappa}}-\left(\eta_{n}-\alpha\right)>0.
\]

Therefore, we have $\kappa h_{i}^{\kappa-1}(\tau)-(\eta_{i}-\alpha)>\chi_{i}(\alpha)>0$. Now, for $l=2,3,\ldots,2^{n}$ and $\tau \in(0, A)$, if $G_{l}(\tau)=0$, then using condition \eqref{eq3.13}:
\[
G_{l}^{\prime}(\tau)<\sum_{i=2}^{n} \frac{\alpha}{\chi_{i}(\alpha)}-\frac{\alpha}{(1-\kappa)\left(\eta_{n}-\alpha\right)}-1 \leq 0.
\]

This proves that for each $l=2,3,\ldots,2^{n}$, there exists exactly one value of $\tau \in(0, A)$ such that $G_{l}(\tau)=0$, and consequently, system \eqref{eq3.2} has exactly $2^{n}-1$ positive solutions.
\end{proof}
The conditions of this lemma are satisfied in particular for $\alpha>0$ such that $\alpha<\frac{1}{2}\eta_{1}$, $\alpha \leq \frac{1}{n-1}\eta_{n}$, and
\[
\alpha \leq \frac{4s}{(N-2s)(2-p)}\left(\frac{2-p}{q}\right)^{\frac{(N-2s)q}{4s}}\left(2\eta_{n}\right)^{1-\frac{(N-2s)q}{4s}}\left(\sum_{i=1}^{n}\left(\eta_{i}-\frac{1}{2}\eta_{1}\right)^{-\frac{(N-2s)q}{4s}}\right)^{-1}.
\]
since, according to this last inequality, we can verify that $\chi_{i}(\alpha) \geq \eta_{n}$ for all indices $i$.

Under the assumptions of Lemmas \ref{lem3.7} and \ref{lem3.8}, we have established that for $l=2,3,\ldots,2^{n}$ and $\tau \in(0, A)$, $G_{l}^{\prime}(\tau)<0$ whenever $G_{l}(\tau)=0$. The uniqueness of solutions to the equation $G_{l}(\tau)=0$ is a direct consequence of this property. It is important to note that it is generally impossible to have $G_{l}^{\prime}(\tau)<0$ for all $\tau \in(0, A)$. Indeed, if $n \in I \backslash I_{l}$ and $\eta_{n-1}<\eta_{n}$, then $\lim_{\tau \rightarrow A^{-}} G_{l}^{\prime}(\tau)=+\infty$, since as $\tau$ approaches $A$ from the left, the term $\alpha h_{n}^{\prime}(\tau)$ in the expression for $G_{l}^{\prime}(\tau)$ tends to $+\infty$ while all other terms converge to finite values.

\begin{proof}[\bf{Part (c)}] The result follows from Lemmas \ref{lem3.1}, \ref{lem3.6}, \ref{lem3.7} and \ref{lem3.8}.
\end{proof}
\begin{proof}[\bf{Part (d)}] We now establish the uniqueness result. From our assumptions, we have $\alpha > \eta^{\prime\prime} \geq \eta^{\prime}$. According to Lemma \ref{lem3.3}, system \eqref{eq3.1} has a unique positive solution, which we denote by $(t_1, t_2, \ldots, t_n)$. Let $(u_1, u_2, \ldots, u_n)$ represent any positive solution of system \eqref{eq:1.3b}, and define $k_j = t_j^{1/q}$ and $U_j = k_j^{-1}u_j$. The existence of such a solution $(u_1, u_2, \ldots, u_n)$ is guaranteed by Lemma \ref{lem3.3}.

Given that $\eta_1 = \cdots = \eta_m = \eta^{\prime}$ and $\eta_{m+1} = \cdots = \eta_{2m} = \eta^{\prime\prime}$, we can readily observe that $t_1 = \cdots = t_m$ and $t_{m+1} = \cdots = t_{2m}$. Consequently, $k_1 = \cdots = k_m$ and $k_{m+1} = \cdots = k_{2m}$. To prove that $(u_1, u_2, \ldots, u_n)$ is the unique positive solution of system \eqref{eq:1.3b}, it suffices to demonstrate that $U_1 = U_2 = \cdots = U_n = U$.

We first prove that $U_1 = \cdots = U_m$ and $U_{m+1} = \cdots = U_{2m}$. For brevity, we will only establish that $U_1 = U_2$, as the proof for the other equalities follows the same reasoning. We proceed by contradiction: suppose $U_1 \neq U_2$ and define the set $\Omega = \{x \in \mathbb{R}^N \mid U_1(x) > U_2(x)\} \neq \emptyset$. For notational convenience, we denote $\eta := \eta^{\prime}$, $s_1 = s_2 = s^{\prime}$, and $t := t_1 = t_2$. Under these conditions, $U_1$ and $U_2$ satisfy the system:
\[
\left\{\begin{array}{l}
(-\Delta)^{s^{\prime}} U_1 = t^{-\kappa}\left(\eta t U_1^{2^*_s-1} + \alpha t U_1^{p-1}U_2^q + \alpha \sum_{j=3}^n t_j U_1^{p-1}U_j^q\right) ,\\
(-\Delta)^{s^{\prime}} U_2 = t^{-\kappa}\left(\eta t U_2^{2^*_s-1} + \alpha t U_1^q U_2^{p-1} + \alpha \sum_{j=3}^n t_j U_2^{p-1}U_j^q\right).
\end{array}\right.
\]

Multiplying the first equation by $U_2$, the second equation by $U_1$, and integrating over $\Omega$ yields:
\[
\begin{aligned}
\int_{\Omega}\left[\left((-\Delta)^{s^{\prime}} U_1\right)U_2 + \left((-\Delta)^{s^{\prime}} U_2\right)U_1\right] = &\; t^{-\kappa}\int_{\Omega}\left(\eta t U_1^{2^*_s-1}U_2 + \alpha t U_1^{p-1}U_2^{q+1} - \eta t U_2^{2^*_s-1}U_1 - \alpha t U_1^{q+1}U_2^{p-1}\right) \\
&+ t^{-\kappa}\alpha \sum_{j=3}^n t_j \int_{\Omega}U_j^q U_1 U_2\left(U_1^{p-2} - U_2^{p-2}\right) \triangleq I_1 + I_2.
\end{aligned}
\]

For the left-hand side of this equation, we have:
\[
\mathrm{LHS} \geq \int_{\mathbb{R}^N \setminus \Omega}(-\mathcal{N}_{s^{\prime}}U_1 + \mathcal{N}_{s^{\prime}}U_2)U_1 \geq 0.
\]
where $\mathcal{N}_{s^{\prime}}u(x) = C_{N,s^{\prime}}\int_{\mathbb{R}^N \setminus \Omega}\frac{u(x)-u(y)}{|x-y|^{n+2s^{\prime}}}\mathrm{d}y, ~x \in \mathbb{R}^N \setminus \Omega$ represents the nonlocal normal derivative.

On the right-hand side, we have two terms. The second term $I_2 < 0$, since $p < 2$ and $U_1 > U_2$ on $\Omega$ imply that the integrand of $I_2$ is negative throughout $\Omega$. To analyze the integral in $I_1$, we split and recombine the four terms of its integrand as follows:
\[
\begin{aligned}
&\eta t U_1^{2^*_s-1}U_2 + \alpha t U_1^{p-1}U_2^{q+1} - \eta t U_2^{2^*_s-1}U_1 - \alpha t U_1^{q+1}U_2^{p-1} \\
&\quad = \eta t U_1^{p-1}U_2\left(U_1^q - U_2^q\right) + (\alpha + \eta)t U_1^{p-1}U_2^{q+1} + \eta t U_1 U_2^{p-1}\left(U_1^q - U_2^q\right) - (\alpha + \eta)t U_1^{q+1}U_2^{p-1}.
\end{aligned}
\]

Since $U_1 > U_2$ on $\Omega$ and $\eta < \frac{\alpha + \eta}{2}$, we can rearrange the terms on the right-hand side and factorize to obtain:
\[
\begin{aligned}
&\eta t U_1^{2^*_s-1}U_2 + \alpha t U_1^{p-1}U_2^{q+1} - \eta t U_2^{2^*_s-1}U_1 - \alpha t U_1^{q+1}U_2^{p-1} \\
&\quad < \frac{\alpha + \eta}{2}t\left(U_1^{q+1}U_2 + U_1 U_2^{q+1}\right)\left(U_1^{p-2} - U_2^{p-2}\right) < 0.
\end{aligned}
\]

Thus, $I_1 < 0$. This leads to a contradiction: $0 \leq I_1 + I_2 < 0$. Therefore, $U_1 = \cdots = U_m$ and $U_{m+1} = \cdots = U_{2m}$.

Now we establish that $U_1 = U_{m+1}$. Having already proven that $U_1 = \cdots = U_m$ and $U_{m+1} = \cdots = U_{2m}$, we observe that $U_1$ and $U_{m+1}$ satisfy the system:
\[
\left\{\begin{array}{l}
(-\Delta)^{s^{\prime}} U_1 = t_1^{1-\kappa}(\eta^{\prime} + (m-1)\alpha)U_1^{2^*_s-1} + m\alpha t_1^{-\kappa}t_{m+1}U_1^{p-1}U_{m+1}^q ,\\
(-\Delta)^{s^{\prime}}U_{m+1} = t_{m+1}^{1-\kappa}(\eta^{\prime\prime} + (m-1)\alpha)U_{m+1}^{2^*_s-1} + m\alpha t_1 t_{m+1}^{-\kappa}U_1^q U_{m+1}^{p-1}.
\end{array}\right.
\]

We proceed by contradiction. If $U_1 \neq U_{m+1}$, let us assume that the set $\Omega = \{x \in \mathbb{R}^N \mid U_1(x) > U_{m+1}(x)\}$ is non-empty. Integrating over $\Omega$, we obtain:
\[
\begin{aligned}
\int_{\Omega}[&((-\Delta)^{s^{\prime}}U_1)U_{m+1} + ((-\Delta)^{s^{\prime}}U_{m+1})U_1] = \int_{\Omega}(t_1^{1-\kappa}(\eta^{\prime} + (m-1)\alpha)U_1^{2^*_s-1}U_{m+1} + m\alpha t_1^{-\kappa}t_{m+1}U_1^{p-1}U_{m+1}^{q+1} \\
&- t_{m+1}^{1-\kappa}(\eta^{\prime\prime} + (m-1)\alpha)U_1 U_{m+1}^{2^*_s-1} - m\alpha t_1 t_{m+1}^{-\kappa}U_1^{q+1}U_{m+1}^{p-1}).
\end{aligned}
\]

For the left-hand side of this equation, we have:
\[
\mathrm{LHS} \geq \int_{\mathbb{R}^N \setminus \partial\Omega}(-\mathcal{N}_{s^{\prime}}U_1 + \mathcal{N}_{s^{\prime}}U_{m+1})U_1 \geq 0.
\]

Let us denote by $G$ the integrand on the right-hand side, and reorganize its four terms as:
\[
\begin{aligned}
G = &\, t_1^{1-\kappa}(\eta^{\prime} + (m-1)\alpha)U_1^{p-1}U_{m+1}(U_1^q - U_{m+1}^q) + U_1^{p-1}U_{m+1}^{q+1} \\
&+ t_{m+1}^{1-\kappa}(\eta^{\prime\prime} + (m-1)\alpha)U_1 U_{m+1}^{p-1}(U_1^q - U_{m+1}^q) - U_1^{q+1}U_{m+1}^{p-1},
\end{aligned}
\]
where, according to system \eqref{eq3.1}, $t_1$ and $t_{m+1}$ satisfy:
\begin{equation}\label{eq3.14}
(\eta^{\prime} + (m-1)\alpha)t_1^{1-\kappa} + m\alpha t_1^{-\kappa}t_{m+1} = 1 = (\eta^{\prime\prime} + (m-1)\alpha)t_{m+1}^{1-\kappa} + m\alpha t_1 t_{m+1}^{-\kappa}.
\end{equation}

From the conditions $\eta^{\prime} \leq \eta^{\prime\prime} < \alpha$ and $(\alpha - \eta^{\prime})t_1 + t_1^{\kappa} = (\alpha - \eta^{\prime\prime})t_{m+1} + t_{m+1}^{\kappa}$, we can deduce:
\begin{equation}\label{eq3.15}
t_1 \leq t_{m+1} \quad \text{and} \quad (\alpha - \eta^{\prime\prime})t_{m+1} \leq (\alpha - \eta^{\prime})t_1.
\end{equation}

Applying the assumption from Theorem \ref{thm:2.2}(d) that $\alpha^2 - ((m+1)\eta^{\prime\prime} - (m-1)\eta^{\prime})\alpha + \eta^{\prime}\eta^{\prime\prime} > 0$, we obtain:
\begin{equation}\label{eq3.16}
\frac{\eta^{\prime\prime} + (m-1)\alpha}{m\alpha} < \frac{\alpha - \eta^{\prime\prime}}{\alpha - \eta^{\prime}}.
\end{equation}

Combining the second inequality in \eqref{eq3.15} with \eqref{eq3.16}, we derive:
\begin{equation}\label{eq3.17}
(\eta^{\prime\prime} + (m-1)\alpha)t_{m+1} < m\alpha t_1.
\end{equation}

This inequality, together with the second equality in \eqref{eq3.14}, implies $(\eta^{\prime\prime} + (m-1)\alpha)t_{m+1}^{1-\kappa} < \frac{1}{2}$. Since $\eta^{\prime} \leq \eta^{\prime\prime}$ and $t_1 \leq t_{m}$, by \eqref{eq3.17} we have $(\eta^{\prime} + (m-1)\alpha)t_1 < m\alpha t_{m+1}$, which combined with the first equality in \eqref{eq3.14} yields $(\eta^{\prime} + (m-1)\alpha)t_1^{1-\kappa} < \frac{1}{2}$. 

Incorporating these inequalities into the expression for $G$, we obtain for all $x \in \Omega$:
\[
G < \frac{1}{2}(U_1^{q+1}U_{m+1} + U_1 U_{m+1}^{q+1})(U_1^{p-2} - U_{m+1}^{p-2}) < 0.
\]

This leads to a contradiction: $0 \leq \int_{\Omega}G < 0$. Therefore, $U_1 = U_{m+1}$, which, combined with our previous results, establishes that $U_1 = U_2 = \cdots = U_n = U$, proving the uniqueness of positive solutions to system \eqref{eq:1.3b}.
\end{proof}
\section{Proof of Theorem \ref{thm:2.3}}

In this section, we prove Theorem \ref{thm:2.3}. Throughout, we assume $N > 2s$, $\eta_i > 0$, $p_{ij} = 2$, $q_{ij} = 2^*_s - 2$, and $\alpha_{ij} = \alpha$ for all $i, j \in \{1, 2, \ldots, n\}$ with $i \neq j$. Recall that $\eta_1 \leq \eta_2 \leq \cdots \leq \eta_n$. 

Note that $(k_1 U, k_2 U, \ldots, k_n U)$ is a synchronized positive solution of system \eqref{eq:1.3b} if and only if $(k_1, k_2, \ldots, k_n)$ is a positive solution of the algebraic system:
\[
\eta_i k_i^{2^*_s-2} + \alpha \sum_{j=1, j \neq i}^n k_j^{2^*_s-2} = 1, \quad i = 1, 2, \ldots, n.
\]

\begin{proof}[\bf{Part (a)}]
We can rewrite the system in the following equivalent form:
\begin{equation}\label{eq4.1}
1+(\alpha-\eta_1)k_1^{2^*_s-2} = 1+(\alpha-\eta_2)k_2^{2^*_s-2} = \cdots = 1+(\alpha-\eta_n)k_n^{2^*_s-2} = \alpha\sum_{j=1}^{n}k_j^{2^*_s-2}.
\end{equation}

For this system to admit a positive solution, one of the following conditions must be satisfied: $\alpha > \eta_n$, $0 < \alpha < \eta_1$, or $\alpha = \eta_1 = \eta_n$. Conversely, if either $\alpha > \eta_n$ or $0 < \alpha < \eta_1$, then the system \eqref{eq4.1} has a unique positive solution $(k_1, k_2, \ldots, k_n)$ given by:
\[
k_i = \left((\alpha-\eta_i)\left(\sum_{j=1}^{n}\frac{\alpha}{\alpha-\eta_j}-1\right)\right)^{\frac{-(N-2s)}{4s}},
\]

Furthermore, if $\alpha = \eta_1 = \eta_n$, then any vector $(k_1, k_2, \ldots, k_n)$ with positive components satisfying $\alpha\sum_{j=1}^{n}k_j^{2^*_s-2} = 1$ constitutes a solution.
\end{proof}
\begin{proof}[\bf{Part (b)}]
Consider the case where $\eta_1 \leq \alpha \leq \eta_n$ and $\eta_1 \neq \eta_n$. We can assume there exists some index $i \in \{1, 2, \ldots, n-1\}$ such that $\eta_i \leq \alpha \leq \eta_{i+1}$, with either $\eta_i < \alpha$ or $\alpha < \eta_{i+1}$. Suppose, for contradiction, that system \eqref{eq:1.3b} has a positive solution $(u_1, u_2, \ldots, u_n)$. If we subtract the $(i+1)$-th equation multiplied by $u_i$ from the $i$-th equation multiplied by $u_{i+1}$ and integrate, we arrive at the contradiction:
$$
0 = \int_{\mathbb{R}^N}((\eta_i-\alpha)u_i^{2^*_s-1}u_{i+1} + (\alpha-\eta_{i+1})u_i u_{i+1}^{2^*_s-1}) < 0.
$$
\end{proof}
\begin{proof}[\bf{Part (c)}]
Let $(u_1, u_2, \ldots, u_n)$ be any positive solution of \eqref{eq:1.3b} and $(k_1, k_2, \ldots, k_n)$ be the unique positive solution of equation \eqref{eq4.1}. Define $U_i = \frac{1}{k_i}u_i$ and $t_i = k_i^{2^*_s-2}$. To establish uniqueness, it suffices to prove that $U_1 = U_2 = \cdots = U_n = U$.

We proceed by contradiction. Suppose there exists a set $\Omega = \{x \in \mathbb{R}^N \mid U_1(x) > U_2(x)\} \neq \emptyset$. From the first two equations of system \eqref{eq:1.3b}:
\[
\left\{\begin{array}{l}
(-\Delta)^{s'}U_1 = \eta_1 t_1 U_1^{2^*_s-1} + \alpha t_2 U_1 U_2^{2^*_s-2} + \alpha\sum_{j=3}^{n}t_j U_1 U_j^{2^*_s-2} \\
(-\Delta)^{s'}U_2 = \eta_2 t_2 U_2^{2^*_s-1} + \alpha t_1 U_1^{2^*_s-2}U_2 + \alpha\sum_{j=3}^{n}t_j U_2 U_j^{2^*_s-2}
\end{array}\right.
\]

Multiplying the first equation by $U_2$, the second by $U_1$, and integrating over $\Omega$, we obtain:
\[
\int_{\Omega}[(-\Delta)^{s'}U_1)U_2 + (-\Delta)^{s'}U_2)U_1] = \int_{\Omega}(\eta_1 t_1 U_1^{2^*_s-1}U_2 + \alpha t_2 U_1 U_2^{2^*_s-1} - \eta_2 t_2 U_1 U_2^{2^*_s-1} - \alpha t_1 U_1^{2^*_s-1}U_2).
\]

Since $\alpha > \eta_1$, this leads to the contradiction:
\[
0 \leq \int_{\mathbb{R}^N\setminus\Omega}(-\mathcal{N}_{s'}U_1 + \mathcal{N}_{s'}U_2)U_1 = \int_{\Omega}t_1(\eta_1-\alpha)U_1 U_2(U_1^{2^*_s-2}-U_2^{2^*_s-2}) < 0.
\]
This completes the proof.
\end{proof}
\section{Proof of Theorem \ref{thm:2.4}}
Throughout this section we assume $N >2s$, $\eta_{i}>0$, $2<p_{i j}<2^{*}_{s}$, $p_{i j}+q_{i j}=2^{*}_{s}$, and $\alpha_{i j}>0$.

Consider the algebraic system in equation \eqref{eq2.1}:
\[
f_i(k_1, k_2, \ldots, k_n) := \eta_i k_i^{2^*_s-2} + \sum_{j=1, j\neq i}^{n}\alpha_{ij}k_i^{p_{ij}-2}k_j^{q_{ij}} - 1 = 0, \quad i=1,2,\ldots,n.
\]

Since $2 < p_{ij} < 2^*_s$ and $p_{ij} + q_{ij} = 2^*_s$, for sufficiently small $\varepsilon > 0$, we have:
\[
f_i(k_1, \ldots, \varepsilon, \ldots, k_n) < 0 < f_i(k_1, \ldots, \eta_i^{-(N-2s)/4s}, \ldots, k_n)
\]
for all $k_j \in [\varepsilon, \eta_j^{-(N-2s)/4s}]$ with $j \neq i$ and all $i = 1, 2, \ldots, n$. This implies that the Brouwer degree:
\[
\operatorname{deg}(f, \Omega, 0) = 1,
\]
where $\Omega$ is the $n$-dimensional cuboid defined as $\Omega := \prod_{i=1}^{n}(\varepsilon, \eta_i^{-(N-2s)/4s})$. This guarantees the existence of a synchronized positive solution of system \eqref{eq:1.3b}.

\section{Proof of Theorem \ref{thm:2.5}}

In this section, we proceed with the following assumptions: $N > 2s$, $\eta_i > 0$, $\alpha_{ij} > 0$, $p_{ij} = p \in (2, 2^*_s)$, $q_{ij} = q = 2^*_s - p$, and $\alpha_{ij} = \alpha$ for all distinct indices $i, j \in \{1, 2, \ldots, n\}$. We maintain the notation from Section 4, though with some contextual adjustments.

We define $\kappa = \frac{2-p}{q}$. In contrast to Section 4 where $\kappa \in (0,1)$, the current context has $\kappa \in (-\infty, 0)$. Following the results from Section 4, we know that number of synchronized positive solutions of system \eqref{eq:1.3b} is equivalent to number positive solutions $(t_1, \ldots, t_n, \tau)$ of system \eqref{eq3.2}.

For each $i \in \{1, 2, \ldots, n\}$, let $f_i(t)$ be defined for $t \in (0, +\infty)$ as specified in Section 4. Assume that $\eta_1 \leq \eta_2 \leq \cdots \leq \eta_n$. When $\alpha \leq \eta_1$, each function $f_i$ is strictly decreasing on the interval $(0, S_i)$ with range $(0, +\infty)$. This allows us to define an inverse decreasing function $h_i: (0, +\infty) \rightarrow (0, S_i)$ for each $f_i|_{(0, S_i)}$, where:
\[
S_i = (\eta_i-\alpha)^{-\frac{1}{1-\kappa}} = (\eta_i-\alpha)^{-\frac{(N-2s)q}{4s}},
\]
when $\alpha < \eta_i$, and $S_i = +\infty$ when $\alpha = \eta_i$. Under these conditions, determining the number of synchronized positive solutions for system \eqref{eq:1.3b} reduces to finding positive solutions of the following single algebraic equation:
\begin{equation}\label{eq5.1}
G_1(\tau) := \alpha\sum_{i=1}^{n}h_i(\tau) - \tau = 0, \quad \tau \in (0, +\infty).
\end{equation}

The function $G_1$ exhibits strict monotonicity, specifically decreasing behavior. Additionally, we observe the limiting behaviors:
\[
\lim _{\tau \rightarrow 0^{+}} G_{1}(\tau)=\alpha \sum_{i=1}^{n} S_{i}>0, \quad \lim _{\tau \rightarrow+\infty} G_{1}(\tau)=-\infty,
\]
These properties guarantee that equation \eqref{eq5.1} possesses exactly one solution. This leads us to the following lemma.

\begin{lemma}\label{lem5.1}
If $\alpha \leq \eta_{1}$ then \eqref{eq:1.3b} has exactly one synchronized positive solution.
\end{lemma}

Let us now examine the scenario where $\alpha > \eta_1$. Under this condition, the function $f_1$ attains its minimum value:
\[
A:=\min _{0<t<+\infty} f_{1}(t)=\frac{(-\kappa)^{\frac{\kappa}{1-\kappa}}+(-\kappa)^{\frac{1}{1-\kappa}}}{\left(\alpha-\eta_{1}\right)^{\frac{\kappa}{1-\kappa}}}=\frac{((p-2) / q)^{(N-2s)(2-p) / 4s}+((p-2) / q)^{(N-2s) q / 4s}}{\left(\alpha-\eta_{1}\right)^{(N-2s)(2-p) / 4s}}
\]
This minimum occurs at the point:
\[
T:=\left(\frac{-\kappa}{\alpha-\eta_{1}}\right)^{\frac{1}{1-\kappa}}=\left(\frac{p-2}{q\left(\alpha-\eta_{1}\right)}\right)^{(N-2s) q / 4s}.
\]

For every index $i$, we can uniquely determine a value $T_i'$ with the properties:
\[
0 < T_i' \leq T, \quad f_i(T_i') = A,
\]
The restriction $f_i|_{(0, T_i']}$ is strictly decreasing, mapping $(0, T_i']$ onto $[A, +\infty)$. We define $h_i: [A, +\infty) \rightarrow (0, T_i']$ as the inverse decreasing function of this restriction.

Furthermore, for any $i$ where $\alpha > \eta_i$, there exists a unique second value $T_i''$ satisfying:
\[
T \leq T_i'', \quad f_i(T_i'') = A,
\]
In this case, $f_i|_{[T_i'', +\infty)}$ is strictly increasing from $[T_i'', +\infty)$ onto $[A, +\infty)$. We denote by $k_i: [A, +\infty) \rightarrow [T_i'', +\infty)$ the inverse increasing function of this restriction.

Several important observations: $T_1' = T_1'' = T$; all functions $h_i$ (for $i = 1, 2, \ldots, n$) are well-defined; and the function $k_i$ is well-defined if and only if $\alpha > \eta_i$. Moreover, for any $\tau \in [A, +\infty)$:
\[
h_n(\tau) \leq \cdots \leq h_2(\tau) \leq h_1(\tau) \leq h_1(A) = T.
\]

Additionally, whenever $\alpha > \eta_i$:
\[
k_i(\tau) \geq \cdots \geq k_2(\tau) \geq k_1(\tau) \geq k_1(A) = T.
\]

The graphical representations of functions $f_1$ and $f_i$ when $\alpha > \eta_i$ can be found in Figure 2 of \cite{jing2024number}.

We now introduce the parameter $\rho^*$. As in Section 3, let $I = \{1, 2, \ldots, n\}$. Define $j$ as the maximum index satisfying $\eta_j < \alpha$, and $k$ as the maximum index for which $\eta_1 = \eta_2 = \cdots = \eta_k$ (with the constraint $k \leq j$). Define $J_1, J_2, \ldots, J_{2^{j-k}}$ to be all possible subsets of the index range $\{k+1, k+2, \ldots, j\}$. Note that when $k = j$, this index range becomes empty, leaving only $J_1$ (the empty set).

We define $\rho^*$ as the number of index sets $J_l$ that satisfy:
\[
\alpha \sum_{i \in I \backslash J_{l}} h_{i}(A) + \alpha \sum_{i \in J_{l}} k_{i}(A) = A
\]

This $\rho^*$ represents the number of positive solutions $(t_1, \ldots, t_n, \tau)$ to system \eqref{eq3.2} where $\tau = A$.

The set $\{1, 2, \ldots, j\}$ contains $2^j$ subsets, which we denote as $I_1, I_2, \ldots, I_{2^j}$. For convenience in later arguments, we set $I_1 = \emptyset$. For each $l \in \{1, 2, \ldots, 2^j\}$, we define $\rho_l$ as the number of solutions to:
\[
\alpha \sum_{i \in I \backslash I_{l}} h_{i}(\tau) + \alpha \sum_{i \in I_{l}} k_{i}(\tau) = \tau, \quad \tau \in (A, +\infty).
\]

We define $\rho^{**} = \sum_{l=1}^{2^j} \rho_l$. This value $\rho^{**}$ counts the positive solutions $(t_1, \ldots, t_n, \tau)$ of system \eqref{eq3.2} where $\tau > A$.

The preceding analysis establishes the following lemma:

\begin{lemma}\label{lem5.2}  
For parameters satisfying $\alpha > \eta_1$, the total number of synchronized positive solutions to system \eqref{eq:1.3b} is given by $\rho^* + \rho^{**}$.  
\end{lemma}  

We proceed to establish the validity of Theorem \ref{thm:2.5}.  

\begin{proof}[\bf{Part (a)}]  
When $\alpha \leq \eta_1$, the conclusion follows directly from Lemma \ref{lem5.1}. Now consider the case $\eta_1 < \alpha < \eta_2$. For $\tau \in [A, +\infty)$, define the auxiliary functions:  
\[  
G_1(\tau) := \alpha \sum_{i=1}^n h_i(\tau) - \tau, \quad  
G_2(\tau) := \alpha \sum_{i=2}^n h_i(\tau) + \alpha k_1(\tau) - \tau.  
\]  
By Lemma \ref{lem5.2}, the total solutions of \eqref{eq:1.3b} correspond to the roots of $G_1(\tau) = 0$ over $[A, +\infty)$ and the roots of $G_2(\tau) = 0$ over $(A, +\infty)$. Observe that $G_1(A) = G_2(A)$, and $G_1(\tau)$ is strictly decreasing on $[A, +\infty)$ with $\lim_{\tau \to +\infty} G_1(\tau) = -\infty$. For large $\tau$, the asymptotic behavior of $G_2(\tau)$ is:  
\[  
G_2(\tau) \approx \frac{\alpha}{\alpha - \eta_1} \tau - \tau = \frac{\eta_1}{\alpha - \eta_1} \tau,  
\]  
yielding $\lim_{\tau \to +\infty} G_2(\tau) = +\infty$.  

We assert the existence of $\delta_0 \in (0, \eta_2 - \eta_1)$ such that $G_2$ is strictly increasing on $[A, +\infty)$ for $\eta_1 < \alpha < \eta_1 + \delta_0$. Under this condition:

Case 1:  If $G_1(A) \geq 0$, then $G_1(\tau) = 0$ has one solution in $[A, +\infty)$ and $G_2(\tau) = 0$ has none.

Case 2:  If $G_1(A) < 0$, then $G_1(\tau) = 0$ has no solutions, while $G_2(\tau) = 0$ has exactly one.  

Thus, \eqref{eq:1.3b} admits exactly one synchronized positive solution for $\eta_1 < \alpha < \eta_1 + \delta_0$. Combining this with Lemma \ref{lem5.1} completes the proof of Theorem \ref{thm:2.4}(b).  

Indeed, for $\tau \in(A,+\infty)$ and $i \geq 2$, since $\kappa<0$ and $h_{i}(\tau)<h_{i}(A)<T$, we have, $\kappa h_{i}^{\kappa-1}(\tau)+\alpha-\eta_{i}<$ $\kappa T^{\kappa-1}+\alpha-\eta_{i}<0$. We also have $0<\kappa k_{1}^{\kappa-1}(\tau)+\alpha-\eta_{1}<\alpha-\eta_{1}$. Then, for $\tau \in(A,+\infty)$,
\[
G_{2}^{\prime}(\tau)=\sum_{i=2}^{n} \frac{\alpha}{\kappa h_{i}^{\kappa-1}(\tau)+\alpha-\eta_{i}}+\frac{\alpha}{\kappa k_{1}^{\kappa-1}(\tau)+\alpha-\eta_{1}}-1>-\sum_{i=2}^{n} \frac{\alpha}{\eta_{i}-\eta_{1}}+\frac{\alpha}{\alpha-\eta_{1}}-1 .
\]

From this estimate, it is easy to find a positive number $\delta_{0}$ such that if $\eta_{1}<\alpha \leq \eta_{1}+\delta_{0}$ then for $\tau \in(A,+\infty), G_{2}^{\prime}(\tau)>0$ and thus $G_{2}(\tau)$ is strictly increasing in $[A,+\infty)$.
\end{proof}
Now we prove Theorem \ref{thm:2.5}(b). Assume $\alpha > \eta_1$ and let $j$ denote the maximal integer with $\alpha > \eta_j$. For each $l = 2, \ldots, 2^j$, define:  
\[
G_{l}(\tau):=\alpha \sum_{i \in I \backslash I_{l}} h_{i}(\tau)+\alpha \sum_{i \in I_{l}} k_{i}(\tau)-\tau, \quad \tau \in[A,+\infty),
\]
and we consider the equation $G_{l}(\tau)=0$ for $\tau \in(A,+\infty)$. Since $\emptyset \neq I_{s} \subset\{1, \cdots, j\}$ and for $\tau$ large enough $G_{l}(\tau) \simeq \sum_{i \in I_{l}} \frac{\alpha \tau}{\alpha-\eta_{i}}-\tau$, we have $\lim _{\tau \rightarrow+\infty} G_{l}(\tau)=+\infty$. To achieve the conclusion of Theorem \ref{thm:2.5}(b), we prove that $G_{l}(A)<0$ if $p>2$ and $p$ is sufficiently close to 2 . We have
\[
G_{l}(A)=\alpha \sum_{i \in I \backslash I_{s}} T_{i}^{\prime}+\alpha \sum_{i \in I_{s}} T_{i}^{\prime \prime}-A \leq n \alpha T_{j}^{\prime \prime}-A .
\]
Estimating $A$ and $T_j''$ requires the following lemma:  

\begin{lemma}\label{lem5.3}
Assume $2 < p < 1 + \frac{2^*_s}{2}$. Then
$$
\min\left\{1, \sqrt{\alpha-\eta_1}\right\} \leq A \leq \left(e^{e^{-1}} + 1\right) \max\left\{1, \sqrt{\alpha-\eta_1}\right\}.
$$
\end{lemma}

\begin{proof}
Given $2 < p < 1 + \frac{2^*_s}{2}$, we have $-1 < \kappa = \frac{2-p}{2^*_s-p} < 0$. It is straightforward to check that
$$
1 \leq (-\kappa)^{\frac{\kappa}{1-\kappa}} \leq e^{e^{-1}}, \quad 0 < (-\kappa)^{\frac{1}{1-\kappa}} < 1,
$$
and
$$
\min\left\{1, \sqrt{\alpha-\eta_1}\right\} \leq (\alpha-\eta_1)^{-\frac{\kappa}{1-\kappa}} \leq \max\left\{1, \sqrt{\alpha-\eta_1}\right\}.
$$
Thus, the statement follows from the definition of $A$.

To estimate $T_j''$, note that it depends implicitly on $p$, so we write $T_j'' = T_j''(p)$. The next lemma provides the required estimate.
\end{proof}

\begin{lemma}\label{lem5.4}
For any $\xi \in (0,1)$,
$$
T_j''(p) = O\left((p-2)^{\xi}\right) \quad \text{as} \quad p \to 2^+.
$$
\end{lemma}

\begin{proof}
Suppose $2 < p < 1 + \frac{2^*_s}{2}$. Then $(-\kappa)^{\xi} \in (0,1)$. Since
$$
A = (T_j'')^{\kappa} + (\alpha-\eta_j) T_j'' \geq (\alpha-\eta_j) T_j'',
$$
we get
$$
T_j'' \leq \frac{1}{\alpha-\eta_j} A.
$$
By Lemma~\ref{lem5.3}, we have
$$
T_j'' \leq M := \frac{1}{\alpha-\eta_j} \left(e^{e^{-1}} + 1\right) \max\left\{1, \sqrt{\alpha-\eta_1}\right\}.
$$

Applying Young's inequality, we obtain
$$
A \geq \left( r(\alpha-\eta_j)T_j'' \right)^{1/r} \left( l (T_j'')^{\kappa} \right)^{1/l},
$$
where $r = \frac{1}{(-\kappa)^{\xi}}$ and $l = \frac{1}{1-(-\kappa)^{\xi}}$. Since $\kappa < 0$ and $T_j'' \leq M$, it follows that
$$
A \geq r^{1/r} l^{1/l} (\alpha-\eta_j)^{1/r} M^{\kappa/l} (T_j'')^{1/r}.
$$
Therefore,
$$
T_j'' \leq r^{-1} l^{-r/l} (\alpha-\eta_j)^{-1} M^{-r\kappa/l} A^{r},
$$
which can be rewritten as
$$
T_j'' \leq (\alpha-\eta_j)^{-1} (-\kappa)^{\xi} (1-(-\kappa)^{\xi})^{(-\kappa)^{-\xi}-1} M^{(-\kappa)^{1-\xi}+\kappa} A^{(-\kappa)^{-\xi}}.
$$

Noting the form of $A$, we see
$$
A^{(-\kappa)^{-\xi}} = (-\kappa)^{-(-\kappa)^{1-\xi}/(1-\kappa)} (1-\kappa)^{(-\kappa)^{-\xi}} (\alpha-\eta_1)^{(-\kappa)^{1-\xi}/(1-\kappa)},
$$
and a direct calculation shows $\lim_{\kappa \to 0^-} A^{(-\kappa)^{-\xi}} = 1$. Also,
$$
\lim_{\kappa \to 0^-} (1-(-\kappa)^{\xi})^{(-\kappa)^{-\xi}-1} M^{(-\kappa)^{1-\xi}+\kappa} = e^{-1}.
$$

Since $p \to 2^+$ is equivalent to $\kappa \to 0^-$, as $p \to 2^+$,
$$
T_j''(p) = O\left((-\kappa)^{\xi}\right) = O\left((p-2)^{\xi}\right).
$$
\end{proof}

\begin{proof}[\textbf{Part (b)}]
Fix $\xi \in (0,1)$. By Lemmas~\ref{lem5.3} and~\ref{lem5.4}, there exist $C = C(\alpha) > 0$ and $p_0 = p_0(\alpha) > 2$ such that for $p \in (2, p_0)$ and $l = 2, \ldots, 2^j$,
$$
G_l(A) \leq n\alpha T_j'' - A \leq C(p-2)^{\xi} - \min\left\{1, \sqrt{\alpha-\eta_1}\right\}.
$$
Hence, there is $p_1 = p_1(\alpha) \in (2, p_0)$ so that for $p \in (2, p_1)$ and $l = 2, \ldots, 2^j$, $G_l(A) < 0$. This means that for $l = 2, \ldots, 2^j$, the equation $G_l(\tau) = 0$ has a solution in $(A, +\infty)$, since $G_l(\tau) > 0$ for large $\tau$. By Lemma~\ref{lem5.2}, \eqref{eq:1.3b} has at least $2^j-1$ synchronized positive solutions.
\end{proof}

\begin{proof}[\textbf{Part (c)}]
For $l = 2, 3, \ldots, 2^j$, since $I_l \neq \emptyset$ and $k_i(\tau) \geq k_1(\tau)$,
$$
G_l(\tau) = \alpha \sum_{i \in I \setminus I_l} h_i(\tau) + \alpha \sum_{i \in I_l} k_i(\tau) - \tau > \alpha k_1(\tau) - \tau, \quad \tau \in [A, +\infty).
$$
Since $k_1(\tau) \geq k_1(A) = T$ and $\kappa < 0$, we have $k_1^{\kappa-1}(\tau) \leq T^{\kappa-1} = \frac{\eta_1-\alpha}{\kappa}$. Using $p \geq \frac{\eta_1}{\alpha} 2 + \left(1-\frac{\eta_1}{\alpha}\right) 2^*_s$ yields $\frac{\eta_1-\alpha}{\kappa} \leq \eta_1$. Thus, for $l = 2, \ldots, 2^j$ and $\tau \in [A, +\infty)$,
$$
G_l(\tau) > \alpha k_1(\tau) - \left[ k_1^{\kappa}(\tau) + (\alpha-\eta_1)k_1(\tau) \right] = k_1(\tau)\left(\eta_1 - k_1^{\kappa-1}(\tau)\right) \geq 0.
$$
This shows that for $l = 2, \ldots, 2^j$, $G_l(\tau) = 0$ has no solution in $[A, +\infty)$. Since $G_1(\tau) = \alpha \sum_{i=1}^n h_i(\tau) - \tau$ is strictly decreasing, $\lim_{\tau \to +\infty} G_1(\tau) = -\infty$, and
$$
G_1(A) > \alpha h_1(A) - A = \alpha k_1(A) - A \geq 0,
$$
$G_1(\tau) = 0$ has exactly one solution in $[A, +\infty)$. Therefore, \eqref{eq:1.3b} has a unique synchronized positive solution.
\end{proof}

\section*{Acknowledgment}

A.D. is supported by DST INSPIRE Fellowship with sanction number DST/INSPIRE Fellowship/2022/IF220580. T.M. is supported by CSIR-HRDG grant with grant sanction No. 25/0324/23/EMR-II.


\bibliography{ref}
\bibliographystyle{siam}

 \end{document}